\documentclass{amsart}
\chardef\bslash=`\\

\usepackage{amssymb,mathdots,amsmath,amsfonts,amsthm,epsfig,amscd,stmaryrd,enumitem}
\usepackage{stmaryrd}
\usepackage[all,cmtip,poly]{xy}
\usepackage{xcolor}
\usepackage{tikz}
\usepackage{hyperref}

\usepackage{silence}
\usepackage{microtype}

\newtheorem{theorem}[subsection]{Theorem}
\newtheorem{corollary}[subsection]{Corollary}
\newtheorem{thm}[subsection]{Theorem}
\newtheorem{lemma}[subsection]{Lemma}

\newtheorem{prop}[subsection]{Proposition}
\newtheorem{proposition}[subsection]{Proposition}

\theoremstyle{remark}

\numberwithin{equation}{section}

\newif\iffinalrun
\iffinalrun
\else
 \fi

\iffinalrun
  \newcommand{\need}[1]{}
  \newcommand{\mar}[1]{}
\else
  \newcommand{\need}[1]{{\tiny *** #1}}
  \newcommand{\mar}[1]{\marginpar{\raggedright\tiny  #1}}\fi

\renewcommand\mathbb{\mathbf}

\newcommand{\rec}{{\operatorname{rec}}}

\newcommand{\wv}{{\widetilde{v}}}

\renewcommand{\ell}{l}

\newcommand{\diag}{\operatorname{diag}}
\newcommand{\tr}{\operatorname{tr}}

\newcommand{\A}{\mathbf{A}}
\newcommand{\bA}{\ensuremath{\mathbf{A}}}

\newcommand{\bC}{\ensuremath{\mathbf{C}}}

\newcommand{\bF}{\ensuremath{\mathbf{F}}}

\newcommand{\bQ}{\ensuremath{\mathbf{Q}}}
\newcommand{\Q}{\QQ}
\newcommand{\QQ}{{\mathbb Q}}

\newcommand{\bR}{\ensuremath{\mathbf{R}}}

\newcommand{\bT}{\ensuremath{\mathbf{T}}}

\newcommand{\T}{{\mathbb T}}

\newcommand{\bZ}{\ensuremath{\mathbf{Z}}}

\newcommand{\bbZ}{\ensuremath{\mathbf{Z}}}
\newcommand{\bbQ}{\ensuremath{\mathbf{Q}}}
\newcommand{\cA}{{\mathcal A}}

\newcommand{\cE}{{\mathcal E}}

\newcommand{\cH}{{\mathcal H}}

\newcommand{\cL}{{\mathcal L}}

\newcommand{\cO}{{\mathcal O}}

\newcommand{\cT}{{\mathcal T}}

\newcommand{\ffrm}{{\mathfrak m}}

\DeclareMathOperator{\Aut}{Aut}

\DeclareMathOperator{\End}{End}

\DeclareMathOperator{\Gal}{Gal}
\newcommand{\GL}{\mathrm{GL}}

\DeclareMathOperator{\Hom}{Hom}

\DeclareMathOperator{\nInd}{n-Ind}

\DeclareMathOperator{\SL}{SL}
\DeclareMathOperator{\St}{St}

\newcommand{\Frob}{\mathrm{Frob}}

\newcommand{\Art}{{\operatorname{Art}}}
\newcommand{\Res}{\operatorname{Res}}

\newcommand{\doubleslash}{/\kern-0.2em{/}}

\begin{document}
	
\author{Christos Anastassiades and Jack A. Thorne}
\title[Raising the level for $\GL_{2n}$]{Raising the level of automorphic representations of $\GL_{2n}$ of unitary type}
\begin{abstract}We use the endoscopic classification of automorphic representations of even-dimensional unitary groups to construct level-raising congruences. 
\end{abstract}
\maketitle
\setcounter{tocdepth}{1}
\tableofcontents
\section{Introduction}

\textbf{Context.} The problem of ``level-raising'' was first considered by Ribet \cite{Rib84} in the context of classical modular forms. Let $f$ be a classical modular form of weight 2 and level $\Gamma_0(N)$, which is an eigenform for the Hecke operators. For any choice of rational prime $l$ and isomorphism $\iota : \overline{\bQ}_l \to \bC$, there is an associated continuous semisimple representation
\[ r_\iota(f) : \Gal(\overline{\bQ} / \bQ) \to \GL_2(\overline{\bQ}_l) \]
unramified outside $N l$ and determined, up to isomorphism, by the relation
\[ \det(X - r_\iota(f)(\Frob_p)) = X^2 - \iota^{-1}(a_p) X + p \]
for each prime $p \nmid Nl$; here $\Frob_p$ denotes a geometric Frobenius element and $a_p \in \bC$ the eigenvalue of the Hecke operator $T_p$. 

Let $\overline{r}_\iota(f) : \Gal(\overline{\bQ} / \bQ) \to \GL_2(\overline{\bF}_l)$ denote the semisimple residual representation. Given a prime $q \nmid Nl$, one can ask whether there exists another eigenform $g$ of weight 2 and level $\Gamma_0(Nq)$, new at $q$, and such that $\overline{r}_\iota(g) \cong \overline{r}_\iota(f)$. Ribet showed that the answer is yes, provided that $f$ satisfies certain conditions. A necessary condition is the congruence
\begin{equation}\label{eqn_level_raising_congruence}  \iota^{-1}(a_q) \equiv \pm (q+1) \text{ mod }\ffrm_{\overline{\bZ}_l}. 
\end{equation}
This condition is necessary because there is an isomorphism, a consequence of local-global compatibility,
\[ r_\iota(g)|_{\Gal(\overline{\bQ}_q / \bQ_q)} \cong \left( \begin{array}{cc} \psi & \ast \\ 0 & \epsilon^{-1} \psi \end{array} \right), \]
where $\epsilon$ is the cyclotomic character and $\psi$ is an unramified character such that $\psi^2 = 1$. 

One of the key inputs in \cite{Rib84} is a lemma of Ihara, which describes the kernel of a degeneracy map
\[ H^1(\Gamma_0(N), \bF_l) \oplus H^1(\Gamma_0(N), \bF_l) \to H^1(\Gamma_0(Nq), \bF_l). \]
Diamond and Taylor \cite{Dia94} generalized \cite{Rib84} by considering automorphic forms on quaternion algebras (i.e. inner forms of $\GL_2$ over $\bQ$). In particular, they considered the case of definite quaternion algebras, in which case the analogue of Ihara's lemma can be deduced easily from the strong approximation theorem (see \cite[Lemma 2]{Dia94}). 

Level-raising for automorphic representations of $\GL_n$ when $n > 2$ is less well understood. A natural case to consider is that of regular algebraic automorphic representations of unitary type. By unitary type, we mean automorphic representations $\pi$ of $\GL_n(\bA_E)$, where $E$ is a CM number field, such that the conjugate $\pi^c$ by complex conjugation $c \in \Aut(E)$ is isomorphic to the contragredient $\pi^\vee$. These are the automorphic representations which should be related, by Langlands functoriality, to automorphic representations of unitary groups admitting Shimura varieties (permitting a generalisation of Ribet's technique) or to definite unitary groups (generalising the context of \cite{Dia94}). 

In \cite{cht}, Clozel, Harris and Taylor formulated a conjecture about automorphic forms on definite unitary groups $U_n$ that would play the role, in studying level-raising at a split place  of the quadratic extension defining $U_n$, of Ihara's lemma in the work of Ribet. However, relatively few cases of this conjecture have been proved. We mention the case $n = 2$ (proved in \cite[Lemma 5.3.1]{cht}) and the papers \cite{Tho14, Kar19}, which establish some cases of the conjecture for arbitrary $n$ but under restrictive hypotheses. (One can also study the analogous problem at  places which are not split; see for example the papers \cite{Clo00, Bel06, Sor06}.)

The motivation in \cite{cht} for formulating this conjectural generalization of Ihara's lemma was to be able to prove automorphy lifting theorems. In a subsequent paper \cite{tay}, Taylor gave a method to avoid the use of this conjecture, establishing unconditional automorphy lifting theorems. One can turn the argument around and apply these automorphy lifting results to raise the level of an automorphic representation $\pi$ (always regular algebraic, of unitary type), provided the residual representation $\overline{r}_\iota(\pi)$ is sufficiently non-degenerate (in particular, one usually requires it to be irreducible). See \cite[Theorem 5.1.5]{Gee11} for an example of a result of this type.

\textbf{The results of this paper.} In this paper, we prove new level-raising results for regular algebraic automorphic representations of $\GL_n(\bA_E)$ of unitary type, with only a very weak condition on $\overline{r}_\iota(\pi)$. Our motivation for doing this is applications to the problem of symmetric power fucntoriality for $\GL_2$ (see e.g.\ \cite{New19b}). Here is a special case of our main result, Theorem \ref{thm_level_raising_for_general_linear_groups_two_factor_case}.
\begin{theorem}\label{intro_main_theorem}
Let $E$ be a CM number field, and let $n \geq 1$ be an integer. Let $\pi_1, \pi_2$ be cuspidal automorphic representations of $\GL_n(\bA_E)$ of unitary type such that $\pi = \pi_1 \boxplus \pi_2$ is regular algebraic. Let $l$ be a prime, and let $\iota : \overline{\bQ}_l \to \bC$ be an isomorphism. Suppose that the following conditions are satisfied:
\begin{enumerate}
\item $\overline{r}_\iota(\pi)$ is not isomorphic to any character twist of the representation $1 \oplus \epsilon^{-1} \oplus \dots \oplus \epsilon^{1-2n}$.
\item There exists a place $w \nmid l$ of $E$ and unramified characters $\xi_1, \xi_2 : E_w^\times \to \bC^\times$ satisfying the following:
\begin{enumerate}
\item $\pi_{1, w} \cong \St_n(\xi_1)$ and $\pi_{2, w} \cong \St_n(\xi_2)$, where $\St_n$ denotes the Steinberg representation of $\GL_n(E_w)$.
\item Let $\varpi_w \in E_w$ be a uniformizer. Then there is a congruence 
\[ \iota^{-1} \xi_1(\varpi_w) / \xi_2(\varpi_w) \equiv q_w^{n} \text{ mod }\ffrm_{\overline{\bZ}_l}, \]
\end{enumerate}
where $\ffrm_{\overline{\bZ}_l}$ denotes the maximal ideal of $\overline{\bZ}_l$.
\end{enumerate}
Then we can find the following:
\begin{enumerate}
\item A biquadratic CM extension $E' / E$, and a place $w' | w$ of $E'$.
\item A cuspidal automorphic representation $\Pi$ of $\GL_{2n}(\bA_{E'})$, regular algebraic of unitary type, such that $\overline{r}_\iota(\Pi) \cong \overline{r}_\iota(\pi)|_{\Gal(\overline{E} / E')}$ and $\Pi_{w'}$ is an unramified twist of the Steinberg representation of $\GL_{2n}(E'_{w'})$.
\end{enumerate}
\end{theorem}
The proof of this theorem is based on the endoscopic classification of automorphic representations of $U_{2n}$. More precisely, we choose a unitary group $U_{2n}$ over the maximal totally real subfield $F$ of $E$, an inner form of the quasi-split unitary group $U_{2n}^\ast$ associated to the quadratic extension $E/F$, which has the following properties:
\begin{itemize}
\item $U_{2n}$ is definite: in other words, $U_{2n}(F \otimes_\bQ \bR)$ is compact.
\item Let $v$ denote the place of $F$ below $w$, which we assume splits in $E$. Then there is an isomorphism $U_{2n}(F_v) \cong \GL_2(D_w)$, where $D_w$ is a central division algebra over $E_w$ of rank $n$.
\end{itemize}
This group is useful to us for the following reasons. First, the analogue of Ihara's lemma at the place $v$ follows from strong approximation, just as in the case $n = 1$. Second, the endoscopic classification implies that the automorphic representations of $U_{2n}(\bA_F)$ should admit a description in terms of the regular algebraic automorphic representations of $\GL_{2n}(\bA_E)$ of unitary type. In particular, the multiplicity of (a descent of) $\pi$ in the space of automorphic forms on $U_{2n}$ should be described by Arthur's multiplicity formula, allowing us to reformulate the problem of level-raising for $\GL_{2n}$ in terms of a more accessible one about level-raising for the group $U_{2n}$.

We can now explain the reason we must pass to an extension $E' / E$ in the statement of the Theorem \ref{intro_main_theorem}: it is there so that the multiplicity on $U_{2n}$ is positive! At the time of writing a proof of the endoscopic classification for groups like $U_{2n}$ (which are inner twists, but not pure inner twists, of the quasi-split form) has been announced by Kaletha, Minguez, Shin, and White \cite{Kal14}, but a complete proof has not yet appeared. We therefore establish the small piece of the endoscopic classification that we need using existing references.

\textbf{Organization of this paper.}  We begin in \S \ref{sec_jacquet_langlands} by recalling some basic facts about the relation between representations of $\GL_{2n}(E_w)$ and $\GL_2(D_w)$, and consider as well the local avatar of the level-raising degeneracy operator. In \S \ref{sec_endoscopic_classification}, we state and prove what we require on the endoscopic classification in the cases described above. Our main reference for the trace formula and its stabilization is the paper of Labesse \cite{labesse}. To compute, we rely on Kaletha's formulation of the local Langlands conjectures, as exposed in \cite{Kal16c}.

In \S \ref{sec_congruences_unitary_case}, we prove a version of our main theorem for automorphic representations of the group $U_{2n}$; this requires the analogue of Ihara's lemma, but not yet the endoscopic classification. In \S \ref{sec_congruences_general_linear_case}, we combine the results of \S \S\ref{sec_endoscopic_classification}--\ref{sec_congruences_unitary_case} to prove our main result on automorphic representations of $\GL_{2n}$. 

\textbf{Acknowledgements.} This work is based on the PhD thesis of the first author (C.A.). We thank Laurent Clozel and Tasho Kaletha for useful conversations relating to the arguments in \S \ref{sec_endoscopic_classification}. We thank the anonymous referee for their useful comments.

J.T.'s work received funding from the European Research Council (ERC) under the European Union's Horizon 2020 research and innovation programme (grant agreement No 714405). 

\textbf{Notation.} Throughout this paper, if $K$ is a field, then we write $G_K$ for the Galois group with respect to a fixed separable closure $K^s / K$. If $K$ is a number field and $v$ is a place of $K$ then we write $K_v$ for the completion, and (in case $v$ is a finite place) $\cO_{K_v} \subset K_v$ for the valuation ring, $\varpi_v \in \cO_{K_v}$ for a choice of uniformizer, $k(v) = \cO_{K_v} / (\varpi_v)$ for the residue field, and $q_v = \# k(v)$ for the size of the residue field. We fix embeddings $K^s \to K_v^s$ extending $K \to K_v$, obtaining injections $G_{K_v} \to G_K$. We write $\bA_K$ for the adele ring of $K$, $\bA_K^\infty$ for the ring of finite adeles, and $\widehat{\cO}_K \subset \bA_K^\infty$ for the open subring given by the profinite completion of the ring of integers $\cO_K \subset K$.

We use the notation for the local Langlands correspondence for $\GL_n$ described in \cite[\S 1]{Clo14}. In particular, if $K$ is a non-archimedean local field and $\pi$ is an irreducible admissible representation of $\GL_n(K)$ over a field $\Omega$, abstractly isomorphic to $\bC$, then we use the Tate normalisation to define $\rec^T_K(\pi)$, a Frobenius-semisimple Weil--Deligne representation over $K$. The same reference also recalls the notion of regular algebraic automorphic representation (a condition on the infinite component of an automorphic representation $\pi$ of $\GL_n(\bA_K)$). 

If $E$ is a CM field (i.e.\ a totally imaginary quadratic extension of a totally real number field $F$), then we write $c \in \Gal(E/F)$ for the non-trivial automorphism. We say that an automorphic representation $\pi$ of $\GL_n(\bA_E)$ is conjugate self-dual if $\pi^c \cong \pi^\vee$. We will often use the fact that if $\pi$ is a RACSDC (regular algebraic, conjugate self-dual, cuspidal) automorphic representation of $\GL_n(\bA_E)$, $l$ is a prime, and $\iota : \overline{\bQ}_l \to \bC$ is an isomorphism, then there exists an associated continuous representation $r_\iota(\pi) : G_E \to \GL_n(\overline{\bQ}_l)$, which satisfies the relation
\begin{equation}\label{eqn_lgc}
\operatorname{WD}(r_\iota(\pi)|_{G_{E_v}})^{F-ss} \cong \rec^T_{E_v}(\pi_v)
\end{equation}
for each place $v\nmid l$ of $E$ (see \cite{Caraianilnotp}; in other words, the associated Weil--Deligne representation respects the local Langlands correspondence, after passage to Frobenius-semisimplification). 

This can be slightly extended: if instead $\pi_1, \dots, \pi_k$ are conjugate self-dual, cuspidal, automorphic representations of $\GL_{n_1}(\bA_E), \dots, \GL_{n_k}(\bA_E)$ (therefore unitary), and $\pi = \pi_1 \boxplus \dots \boxplus \pi_k$ is the normalized induction, a conjugate self-dual, unitary automorphic representation of $\GL_{\sum n_i}(\bA_E)$, then if $\pi$ is regular algebraic then there exists a continuous semisimple representation $r_\iota(\pi) : G_E \to \GL_{\sum n_i}(\overline{\bQ}_p)$ satisfying (\ref{eqn_lgc}). This follows easily from the case $k = 1$. We note that for $\pi$ satisfying these conditions, we can define what it means for $\pi$ to be $\iota$-ordinary (\cite[Definition 2.4]{Clo14}).

We will use the existence of cyclic base change. If $E$ is a number field and $\pi_1, \dots, \pi_k$ are unitary cuspidal representations of $\GL_{n_1}(\A_E), \dots, \GL_{n_k}(\A_E)$, then cyclic base change associates to the representation $\pi = \pi_1 \boxplus \dots \boxplus \pi_k$ of $\GL_n(\A_E)$ and any cyclic extension $E' / E$ of number fields the base change representation $\pi_{E'}$ of $\GL_n(\A_{E'})$ (see \cite[Ch. 3, Theorem 4.2, Theorem 5.1]{MR1007299}.

Throughout this paper we use overline to denote semisimple residual representations (e.g.\ $\overline{r}_\iota(\pi) : G_E \to \GL_n(\overline{\bF}_l)$).

\section{Recollections on the Jacquet--Langlands correspondence}\label{sec_jacquet_langlands}

Let $K / \bQ_p$ be a finite extension, and let $D$ be a central division algebra over $K$ of rank $n \geq 1$. We recall that the Jacquet--Langlands correspondence is a bijection $JL$ from the set of (isomorphism classes) of square-integrable irreducible admissible representations of $\GL_{d}(D)$ to the set of square-integrable irreducible admissible representations of $\GL_{dn}(K)$ (see \cite{Del84}). It is uniquely characterized by a character identity, which is generalized below. Badulescu  has defined a map $|LJ_{\GL_d(D)}|$ from the set of irreducible unitary representations of $\GL_{dn}(K)$ to the set of irreducible unitary representations of $\GL_d(D)$, together with the zero representation (see \cite{Bad08}). The map $|LJ_{\GL_d(D)}|$ is neither injective nor surjective in general, but it does satisfy $|LJ_{\GL_d(D)}|(JL(\pi)) = \pi$ if $\pi$ is square-integrable. If $\pi^\ast$ is any irreducible unitary representation of $\GL_{dn}(K)$, and $f \in C_c^\infty(\GL_d(D))$, then we have the identity (cf. \cite[Proposition A.0.1]{Her19}):
\begin{equation}\label{eqn_LJ_character_identity}
\tr \pi^\ast(f^\ast) = (-1)^{d(n-1)} \tr |LJ_{\GL_d(D)}|(\pi^\ast)(f),
\end{equation}
where $f^\ast \in C_c^\infty(\GL_{dn}(K))$ is a stable transfer of $f$ (see \S \ref{sec_endoscopic_classification} for more recollections on this). 

The reduced norm defines a homomorphism $\det : \GL_d(D) \to K^\times$. If $d = 1$, then we denote this as $N : D^\times \to K^\times$.  Both $JL$ and  $|LJ_{\GL_d(D)}|$ are compatible with the operation of twisting by characters of the form $\chi \circ \det$.

We now specialise to the case $d = 2$, which is the one of interest for us. Let $P_0 \subset \GL_2(D)$ denote the minimal parabolic subgroup of $\GL_2(D)$ consisting of upper-triangular matrices with entries in $D$, and let $M_0 = D^\times \times D^\times$ denote its diagonal Levi subgroup. Let $\delta_{P_0} : P_0 \to \bR_{>0}$ denote the character $\delta_{P_0}(\diag(d_1, d_2)) = | N(d_1 d_2^{-1}) |^n$, where $| \cdot | : K^\times \to \bR_{>0}$ is normalized, as usual, so that the norm of a uniformizer equals the reciprocal cardinality of the residue field of $K$.

Let $\cO_D$ denote the ring of integers of $D$, and let $\varpi_D \in \cO_D$ be a fixed choice of uniformizer. We set $\mathfrak{K} = \GL_2(\cO_D)$, and write $\mathfrak{I} \subset \mathfrak{K}$ for the subset of elements whose lower-left entry is divisible by $\varpi_D$. Then $\mathfrak{I}$ is an Iwahori subgroup of $\GL_2(D)$.

The representations of $\GL_2(D)$ that we will be concerned with are the irreducible subquotients of the normalized induction
\[ \nInd_{P_0}^{\GL_2(D)} (\chi_1 \circ N \otimes \chi_2 \circ N), \]
 where $\chi_1, \chi_2 : K^\times \to \bC^\times$ are unramified characters. The properties of these are described by the following proposition, which generalizes well-known facts in the case $D = K$.
\begin{proposition}\label{prop_classification_of_Iwahori_spherical_repns}
\begin{enumerate}
\item Let $\pi$ be an irreducible admissible representation of $\GL_2(D)$. Then $\pi^{\mathfrak{I}} \neq 0$ if and only if $\pi$ is isomorphic to a subquotient of a representation $\nInd_{P_0}^{\GL_2(D)} \chi_1 \circ N \otimes \chi_2 \circ N$, where $\chi_1, \chi_2 : K^\times \to \bC^\times$ are unramified characters.
\item Let $\chi_1, \chi_2 : K^\times \to \bC^\times$ be unramified characters, and let $\pi = \nInd_{P_0}^{\GL_2(D)} \chi_1 \circ N \otimes \chi_2 \circ N$. Then $\pi$ is irreducible if and only if $\chi_1 / \chi_2 \neq | \cdot |^{\pm n}$. If $\chi_1  / \chi_2= | \cdot |^{\pm n}$, then $\pi$ has two irreducible subquotients: the character $\chi_1 \circ \det$, and an essentially square-integrable representation.
\end{enumerate}
\end{proposition}
\begin{proof}
The first part of the proposition follows from well-known results on Jacquet modules, see for example \cite[\S 2.1]{Her19}. The second part follows from the results of \cite[\S 2]{Tad90}.
\end{proof}
We write $\St_{\GL_2(D)}$ for the unique square-integrable subquotient of 
\[ \nInd_{P_0}^{\GL_2(D)} | \cdot |^{n/2} \circ N \otimes | \cdot |^{-n/2} \circ N, \]
 and call it the Steinberg representation of $\GL_2(D)$. If $\chi : K^\times \to \bC^\times$ is an unramified character, then we define $\St_{\GL_2(D)}(\chi) = \St_{\GL_2(D)} \otimes (\chi \circ \det)$. The first part of the following lemma justifies our use of language.
\begin{lemma}\label{lem_computation_of_LJ}
\begin{enumerate}
	\item We have $|LJ_{\GL_{2}(D)}|(\St_{2n}) = \St_{\GL_2(D)}$. 
	\item Let $\chi_1, \chi_2 : K^\times \to \bC^\times$ be unitary unramified  characters. Then 
	\[ |LJ_{\GL_{2}(D)}|(\St_n(\chi_1) \boxplus \St_n(\chi_2)) = \nInd_{P_0}^{\GL_2(D)} \chi_1 \circ N \otimes \chi_2 \circ N. \]
	\item Let $\pi$ be an irreducible unitary admissible representation of $\GL_{2n}(K)$. Suppose that there is an unramified character $\chi : K^\times \to \bC^\times$ such that $|LJ_{\GL_2(D)}|(\pi) = \St_{\GL_2(D)}(\chi)$. Then either $\pi \cong \St_{2n}(\chi)$ or $\pi \cong \chi \circ \det$.
\end{enumerate}
\end{lemma}
\begin{proof}
The first part is contained in the statement of \cite[B.2.b, Th\'eor\`eme]{Del84}. The second follows from \cite[Proposition 3.4]{Bad07}. For the third, we can assume (after twisting) that $\chi = 1$. It suffices to show that $\pi$ has the same cuspidal support as $\St_{2n}$; then \cite[2.1, Theorem]{Cas81} shows that $\pi$ must have the claimed form.

To compute the cuspidal support of $\pi$, we write $\pi = \sum_{i \in I} a_i \nInd_{P_i}^{\GL_{2n}(K)} \sigma_i$ as a sum (in the Grothendieck group of irreducible admissible representations of $\GL_{2n}(K)$) of standard representations; thus the $a_i$'s are non-zero integers, the $P_i$'s are standard parabolic subgroups of $\GL_{2n}(K)$, and the $\sigma_i$'s are essentially square-integrable representations of the Levi quotients of the $P_i$'s. Each representation $\sigma_i$ has the same cuspidal support as $\pi$. Then, by definition (see \cite[\S 2.7]{Bad08}), we have 
\[ |LJ_{\GL_2(D)}|(\pi) = \pm \sum_{i \in I'} a_i \nInd_{P'_i}^{\GL_{2}(D)} \sigma'_i, \]
where now $I' \subset I$ is the set of indices for which $P_i$ corresponds to a standard parabolic subgroup $P_i'$ of $\GL_2(D)$ and $\sigma'_i$ is the essentially square-integrable representation of the Levi quotient of $P_i'$ corresponding to $\sigma_i$ under the Jacquet--Langlands correspondence.

By hypothesis, some representation  $\nInd_{P'_i}^{\GL_{2}(D)} \sigma'_i$ contains $\St_{\GL_2(D)}$ as a subquotient. If $P'_i = \GL_2(D)$ then $\sigma'_i = \St_{\GL_2(D)}$, $\sigma_i = \St_{2n}$, and we're done. Otherwise, $P'_i = P_0$ and $\sigma_i' =  | \cdot |^{n/2} \circ N \otimes | \cdot |^{-n/2} \circ N$, in which case $\sigma_i$ is the representation $\St_n(| \cdot |^{n/2}) \otimes \St_n(| \cdot |^{-n/2})$ of $\GL_n(K) \times \GL_n(K)$. We again see that $\sigma_i$ has the same cuspidal support as $\St_{2n}$. This completes the proof.
\end{proof}
Let $\eta = \diag(1, \varpi_D) \in \GL_2(D)$. Note that $ \mathfrak{I} \subset \eta \mathfrak{K} \eta^{-1}$, so if $\pi$ is an admissible representation of $\GL_2(D)$, then $\eta \cdot \pi^{\mathfrak{K}} \subset \pi^{\mathfrak{I}}$.
\begin{lemma}\label{lem_kernel_of_d}
	The subgroup of $\GL_2(D)$ generated by $\GL_2(\cO_D)$ and $\eta \GL_2(\cO_D) \eta^{-1}$ contains the subgroup $\SL_2(D) = \ker( \det: \GL_2(D) \to K^\times)$.
\end{lemma}
\begin{proof}
	We define subgroups
	\[ U^+ = \left\{  \left( \begin{array}{cc} 1 & x \\ 0 & 1 \end{array} \right) \mid x \in D \right \} \text{ and }U^- = \left\{  \left( \begin{array}{cc} 1 & 0 \\ x & 1 \end{array} \right) \mid x \in D \right \}. \]
	We claim that $U^+$ and $U^-$ generate $\SL_2(D)$. Indeed, let $\cT$ denote the subgroup they generate. The relation
	\[ \left( \begin{array}{cc} 1 & 0 \\ x^{-1} - x^{-1} y^{-1} x^{-1}& 1 \end{array}\right) \left( \begin{array}{cc} 1 & -x \\ 0 & 1 \end{array}\right) \left( \begin{array}{cc} 1 & 0 \\ x^{-1} - y&  1 \end{array}\right) \left( \begin{array}{cc} 1 & y^{-1} \\ 0 & 1 \end{array}\right)= \left( \begin{array}{cc} xy & 0 \\ 0 & x^{-1}y^{-1} \end{array}\right),\]
	valid for all $x, y \in D$, shows that $\cT$ contains all elements of the form $\diag(x, x^{-1})$ and $\diag(x y x^{-1}  y^{-1}, 1)$. The commutator subgroup of $D^\times$ equals the kernel of $N$ (see e.g. \cite{Nak43}), so this shows that $\cT$ contains all the diagonal matrices in $\SL_2(D)$.
	
	Now consider an element $g = \left(\begin{array}{cc} a & b \\ c & d \end{array}\right)$ of $\SL_2(D)$. If $d \neq 0$, then we can write
	\[ g = \left(\begin{array}{cc} 1 & b d^{-1} \\ 0 & 1 \end{array}\right) \left(\begin{array}{cc} a - b d^{-1} c & 0  \\ 0 & d \end{array}\right) \left(\begin{array}{cc} 1 & 0  \\ d^{-1} c & 1 \end{array}\right), \]
	showing that $g \in \cT$. If $d = 0$, then $bc \neq 0$, and we can write
	\[ g = \left(\begin{array}{cc} 1 & 0 \\ -1 & 1 \end{array}\right) \left(\begin{array}{cc} a & b \\ a+c & b \end{array}\right), \]
	showing that $g \in \cT$. This completes the proof of our claim that $\cT = \SL_2(D)$.
	
	Since $U^+$ and $U^-$ are conjugate under an element of $\GL_2(\cO_D)$, it now suffices to show that $U^+$ is contained in the subgroup of $\GL_2(D)$ generated by $\GL_2(\cO_D)$ and $\eta \GL_2(\cO_D) \eta^{-1}$. Let $U^+(0) = U^+ \cap \GL_2(\cO_D)$, and let
	\[ \zeta = \left(\begin{array}{cc} 0 & 1\\ 1 & 0 \end{array}\right) \eta \left(\begin{array}{cc} 0 & 1\\ 1 & 0 \end{array}\right) \eta^{-1} = \left(\begin{array}{cc} \varpi_D & 0\\ 0 & \varpi_D^{-1} \end{array}\right).\]
	 Then $U^+ = \cup_{n \geq 0} \zeta^{-n} U^+(0) \zeta^{n}$, so this is true.
\end{proof}
\begin{lemma}\label{lem_characterizing_Steinberg}
	Let $\pi$ be an irreducible admissible representation of $\GL_2(D)$ such that $\pi^{\mathfrak{I}} \neq 0$. Define a map $d : \pi^{\mathfrak{K}} \oplus \pi^{\mathfrak{K}} \to \pi^{\mathfrak{I}}$ by the formula $d(f, g) = f + \eta \cdot g$. If $d$ is not surjective, then there is an unramified character $\chi : K^\times \to \bC^\times$ and an isomorphism $\pi \cong \St_{\GL_2(D)}(\chi)$.
\end{lemma}
\begin{proof}
	Lemma \ref{lem_kernel_of_d} shows that the kernel of $d$ is contained in $\pi^{\SL_2(D)}$, so is trivial if $\pi$ is not 1-dimensional. We will show that if $\pi$ is not a twist of $\St_{\GL_2(D)}$, then $d$ is surjective. Proposition \ref{prop_classification_of_Iwahori_spherical_repns} shows we need only consider two cases. If $\pi = \chi \circ \det$, then $\pi^\mathfrak{K} = \pi^\mathfrak{I}$ and $d$ is surjective. 
	
	If $\pi = \nInd_{P_0}^{\GL_2(D)} \chi_1 \circ N \otimes \chi_2 \circ N$ is an irreducible induced representation, then $\pi^{\mathfrak{K}}$ is 1-dimensional. Indeed, the definition of induction shows that $\pi^{\mathfrak{K}}$ can be identified with the set of functions $\GL_2(D) \to \bC$ such that for all $p \in P_0$, $g \in \GL_2(D)$, $k \in \mathfrak{K}$, $\varphi (pgk) = (\chi_1 \circ N \otimes \chi_2 \circ N)(p) \delta_{P_0}^{1/2}(p) \varphi(g)$. We have $\GL_2(D) = P_0 \mathfrak{K}$, so evaluation at the identity $e \in \GL_2(D)$ defines an isomorphism $\pi^{\mathfrak{K}} \to (\chi_1 \circ N \otimes \chi_2 \circ N)^{\cO_D^\times \times \cO_D^\times}$, and this space is 1-dimensional.
	
	On the other hand, $\pi^{\mathfrak{I}}$ is 2-dimensional (for example, by \cite[Proposition 2.1]{Cas80}). Since $d$ is injective, it is also surjective in this case. This completes the proof.
\end{proof}
In the course of the proof of the last lemma, we showed that if $\chi_1, \chi_2$ are unramified characters and $\pi = \nInd_{P_0}^{\GL_2(D)} \chi_1 \circ N \otimes \chi_2 \circ N$ is irreducible, then $\pi^{\mathfrak{K}}$ is 1-dimensional. In fact, this conclusion holds even when the induced representation is not  irreducible. We conclude this section by computing the eigenvalues of some Hecke operators on the line $\pi^{\mathfrak{K}}$.

Before giving the statement, we recall that the Hecke algebra $\cH(\GL_2(D), \mathfrak{K})$ of $\mathfrak{K}$-biinvariant, compactly supported functions $f : \GL_2(D) \to \bZ$ acts on the space $\pi^{\mathfrak{K}}$. We suppose here that convolution is defined with respect to the Haar measure which gives $\mathfrak{K}$ measure 1.
\begin{lemma}\label{lem_ramified_Hecke_operators}
	Let $T_1 = [ \mathfrak{K} \eta \mathfrak{K} ]$, $T_2 = [\mathfrak{K} \diag(\varpi_D, \varpi_D) \mathfrak{K}] \in \cH(\GL_2(D), \mathfrak{K})$ be the characteristic functions of these double cosets, and let $\pi = \nInd_{P_0}^{\GL_2(D)} \chi_1 \circ N \otimes \chi_2 \circ N$, where $\chi_1, \chi_2 : K^\times \to \bC^\times$ are unramifed characters. Then if $v \in \pi^\mathfrak{K}$, we have
	\[ T_1 v = q_K^{n/2}(\chi_1(\varpi_K) + \chi_2(\varpi_K))v \text{ and }T_2 v = \chi_1 \chi_2(\varpi_K) v. \]
	In particular, $\pi$ is reducible if and only if $\pi^{\mathfrak{K}}$ is annihilated by $T_1^2 - T_2 (q_K^n + 1)^2$.
\end{lemma}  
\begin{proof}
	The computation of Hecke eigenvalues is standard, following the technique of proof of \cite[Lemma 3.1.1]{cht}. The final sentence follows since the eigenvalue of $T_1^2 - T_2 (q_K^n + 1)^2$ is $-(\chi_1(\varpi_K) - q_K^n \chi_2(\varpi_K)) (\chi_2(\varpi_K) - q_K^n \chi_1(\varpi_K))$, and this is zero exactly when $\pi$ is reducible, by Proposition \ref{prop_classification_of_Iwahori_spherical_repns}.
\end{proof}
\section{The endoscopic classification}\label{sec_endoscopic_classification}

In this section we prove what we need concerning the endoscopic classification of automorphic representations of certain unitary groups. The results we prove here are a very special case of the general classification, which has been announced in \cite{Kal14}. However, complete results await a sequel to that paper, so we have chosen to give an unconditional proof of what we need here based on existing references. 

Let $m \geq 1$ be an integer, and let $E / F$ be a quadratic extension of fields of characteristic 0 (inside a fixed algebraic closure $\overline{F} / F$). We define a matrix
\[ \Phi_m = \left( \begin{array}{cccc} 0 & \dots & 0 & -1 \\ 0 & \dots & 1 & 0 \\ \vdots & \iddots & \vdots & \vdots \\
(-1)^m & \dots & 0 & 0 \end{array}\right). \]
If $M = \Res_{E / F} \GL_m$, and $c \in \Gal(E/F)$ denotes the non-trivial automorphism, then we may define an involution $\theta_M : M \to M$ by the formula $\theta_M(g) = \Phi_m {}^t (g^c)^{-1} \Phi_m^{-1}$. The subgroup $U^\ast_m \subset M$ of fixed points is the quasi-split unitary group in $m$ variables. Its functor of points on an $F$-algebra $R$ is
\[ U_m^\ast(R) = \{ g \in M_m(E \otimes_F R) \mid \Phi_m {}^t (g^{c \otimes 1})^{-1} \Phi_m^{-1} = g \}. \]
 We consider $U_m^\ast$ as being endowed with its standard splitting (as described in \cite[\S 0.2.2]{Kal14}). We note that there is a natural identification of the scalar extension $U_{m, E}^\ast$ with $\GL_m$: there is an isomorphism $E \otimes_F E = E \times E$ (where the two $E$-algebra structures agree in the first factor), giving rise to an identification
\[ U_{m, E}^\ast = \{ (g_1, g_2) \in \GL_m \times \GL_m \mid g_2 = \Phi_m g_1^{-t} \Phi_m^{-1} \}, \]
and the projection to the first factor gives the desired isomorphism $U_{m, E}^\ast \to \GL_m$.

\subsection{Statements}\label{sec_endoscopic_classification_setup}

Now let $n \geq 1$ be an integer, let $m = 2n$, and suppose that $E/F$ is a quadratic CM extension of a totally real field. Let $\Sigma$ be a finite set of finite places of $F$, each of which splits in $E$; we fix for each $v \in \Sigma$ a factorization $v = \wv \wv^c$ in $E$, and set $\widetilde{\Sigma} = \{ \wv \mid v \in \Sigma \}$. We assume that $E/F$ is everywhere unramified. Note that this implies that $[F : \bbQ]$ is even. 

We fix the following data:
\begin{itemize}
\item A central division algebra $D$ over $E$ of rank $n$ such that $D \otimes_{E, c} E \cong D^{op}$, such that for each $v \in \Sigma$, $D_\wv$ has invariant $1/n$, and such that for each place $w$ of $E$ not lying above a place of $\Sigma$, $D_w$ is split.
\item An involution $\dagger$ of $B = M_2(D)$ of the second kind (i.e.\ satisfying the conditions $(xy)^\dagger = y^\dagger x^\dagger$ and $\dagger|_E = c$).
\end{itemize}
We define $U_m$ to be the unitary group over $F$ associated to the pair $(B, \dagger)$, i.e.\ with functor of points given on an $F$-algebra $R$ by
\[ U_m(R) = \{ g \in B \otimes_F R \mid g^{\dagger \otimes 1} g = 1. \}\]
We further assume that $\dagger$ is chosen so that $U_m$ satisfies the following conditions:
\begin{itemize}
	\item For each place $v|\infty$ of $F$, $U_m(F_v)$ is compact.
	\item For each finite place $v\not\in \Sigma$ of $F$, $U_{m, F_v}$ is quasi-split.
\end{itemize}
Such a choice exists because $[F : \Q]$ is even. We fix an inner twist $\xi : U_m^\ast \to U_m$ as the composite
\[ \xi : U_{m, \overline{F}}^\ast \cong \GL_{m, \overline{F}} \cong B^\times_{\overline{F}} \cong U_{m, \overline{F}}, \]
where the first and last isomorphisms are the canonical ones and the isomorphism $\GL_{m, \overline{F}} \cong B^\times_{\overline{F}}$ is any fixed choice arising from an isomorphism $M_m(\overline{F}) \cong B \otimes_E \overline{F}$ of central simple $\overline{F}$-algebras. Our choices so far determine the following data:
\begin{itemize}
	\item For each place $v$ of $F$ which splits $v = w w^c$ in $E$, an isomorphism $U_m(F_v) \cong B_w^\times$. If $v \not\in \Sigma$, we write $\iota_w : U_m(F_v) \to \GL_m(E_w)$ for the isomorphism arising from a fixed choice of isomorphism $B_w \cong M_m(E_w)$ of central simple $E_w$-algebras. If $v \in \Sigma$, we write $\iota_w : U_m(F_v) \to \GL_2(D_w)$ for the canonical isomorphism.
	\item For each finite place $v$ of $F$ which is inert in $E$, a $U_m^{\ast, ad}(F_v)$-conjugacy class of isomorphisms $\iota_v : U_m(F_v) \to U_m^\ast(F_v)$.
\end{itemize}
Having fixed the above data we can define a notion of base change relative to the quadratic extension $E / F$. If $\sigma$ is a unitary irreducible admissible representation of $U_m(\bA_F)$, and $\pi$ is a unitary irreducible admissible representation of $\GL_m(\bA_E)$, we will say that $\pi$ is a base change of $\sigma$, and write $\operatorname{BC}(\sigma) = \pi$, if the following conditions hold:
\begin{itemize}
	\item For each inert place $v$ of $F$ at which $\sigma_v$ is unramified (i.e.\ such that $\sigma_v$ has a non-zero invariant vector under some hyperspecial maximal compact subgroup of $U_m(F_v)$), $\pi_v$ is unramified and related to $\sigma_v$ by standard unramified base change (cf. \cite[\S 4.1]{Min11}).
	\item For each split place $v = w w^c$ of $F$ such that $v \not\in \Sigma$, $\pi_w \cong \sigma_v \circ \iota_w^{-1}$.
	\item For each place $v = w w^c$ such that $v \in \Sigma$, $| LJ_{\GL_2(D_w)} |(\pi_w) \cong \sigma_v \circ \iota_w^{-1}$. (In particular, $| LJ_{\GL_2(D_w)} |(\pi_w)$ is non-zero.)
	\item For each archimedean place $v$ of $F$, and each isomorphism $\tau : E_v \to \bC$, let $W_\tau$ be the irreducible algebraic representation of $\GL_m(E_v)$ over $\bC$ such that $\sigma_v \cong W_\tau|_{U_m(F_v)}$. Then $\pi_v$ has the same infinitesimal character as $W(\sigma_v) =  W_\tau \otimes_\bC W_{\tau \circ c}$.
\end{itemize}
We state two theorems about the classification of automorphic representations of $U_m$. The first concerns the existence of base change.
\begin{theorem}\label{thm_unitary_group_base_change}
	Let $\sigma$ be an automorphic representation of $U_m(\bA_F)$. Then there exists a partition $m = m_1 + \dots + m_k$ and discrete automorphic representations $\pi_1, \dots, \pi_k$ of $\GL_{m_1}(\bA_E), \dots, \GL_{m_k}(\bA_E)$, all satisfying the following conditions:
	\begin{itemize}
		\item For each $i = 1, \dots, k$, $\pi_i^c \cong \pi_i^\vee$.
		\item We have $\operatorname{BC}(\sigma) = \pi_1 \boxplus \dots \boxplus \pi_k$.
	\end{itemize}
\end{theorem}
The second concerns the existence of descent. Before formulating it, we observe that if $F ' /F$ is any $\Sigma$-split finite totally real extension, then our hypotheses are still valid over $F'$. More precisely, if $E' = F' E$, then $E' / F'$ is an everywhere unramified CM quadratic extension of a totally real field, the algebras $D \otimes_F F'$ and $B \otimes_E E'$ satisfy the analogous hypotheses relative to $F'$, and give rise to the group $U_{m, F'}$, which comes equipped with an inner twist $\xi' : (U^\ast_{m, F'})_{\overline{F}} \to (U_{m, F'})_{\overline{F}}$.
\begin{theorem}\label{thm_unitary_group_descent}
 Let $\pi_1, \pi_2$ be RACSDC automorphic representations of $\GL_n(\A_E)$ satisfying the following conditions: 
	\begin{enumerate}
		\item $\pi = \pi_1 \boxplus \pi_2$ is regular algebraic.
		\item For each inert place $v$ of $F$, $\pi_v$ is unramified.
		\item For each place $v \in \Sigma$ of $F$ and for each place $w | v$ of $E$, $\pi_{i, w}$ is an unramified twist of the Steinberg representation ($i = 1, 2$).
	\end{enumerate}
	Let $F' / F$ be a $\Sigma$-split, totally real quadratic extension, let $E' = F' E$, and let $\pi_{E'}$ denote the base change of $\pi$ with respect to the quadratic extension $E' / E$. Then there exists an automorphic representation $\sigma$ of $U_m(\bA_{F'})$ which is unramified at the inert places of $E' / F'$, and such that $\operatorname{BC}(\sigma) = \pi_{E'}$. Moreover, $\sigma$ appears with multiplicity 1.
\end{theorem}
Theorem \ref{thm_unitary_group_base_change} is \cite[Proposition 6.5.1]{Her19}. We will prove Theorem \ref{thm_unitary_group_descent} in the next section. Before proceeding to the proof, we record an important consequence of Theorem \ref{thm_unitary_group_base_change}.
\begin{corollary}\label{cor_existence_of_gal_over_Q_l_for_unitary_group}
	Let $\sigma$ be an automorphic representation of $U_m(\bA_F)$, let $l$ be a prime, and let $\iota : \overline{\bQ}_l \to \bC$ be an isomorphism. Then there exists a continuous representation $r_\iota(\sigma) : G_E \to \GL_m(\overline{\bQ}_l)$ satisfying the following conditions: 
	\begin{enumerate} \item Let $v \nmid l$ be an inert place of $F$ at which $\sigma$ is unramified. Then $r_\iota(\sigma)|_{G_{E_v}}$ is unramified.
	\item If $v = w w^c$ is a split place of $F$ and $v \not\in \Sigma$, then $r_\iota(\sigma)|^{F-ss}_{W_{E_w}} \cong \rec_{E_w}^T(\sigma_v \circ \iota_w^{-1})$.
	\end{enumerate}
\end{corollary}
\begin{proof}
	This is a consequence of Theorem \ref{thm_unitary_group_base_change} and known results for $\GL_m(\bA_E)$. Indeed, let $\pi = \pi_1 \boxplus \dots \boxplus \pi_k$ be the base change of $\sigma$. Then each  $\pi_i$ is a discrete automorphic representation of $\GL_{m_i}(\bA_E)$ and $\pi_i | \cdot |^{(m - m_i) / 2}$ is regular algebraic. By the classification of discrete automorphic representations of $\GL_{m_i}(\bA_E)$ \cite{Moe89}, there is a factorisation $m_i = a_i b_i$, a cuspidal automorphic representation $\mu_i$ of $\GL_{a_i}(\bA_E)$ satisfying $\mu_i^c \cong \mu_i^\vee$, and an isomorphism
	\[ \pi_i \cong \mu_i | \cdot |^{(b_i- 1)/2} \boxplus \mu_i | \cdot |^{(b_i- 3)/2} \boxplus \dots \boxplus \mu_i | \cdot |^{(1 - b_i) / 2}. \]
	Moreover, $\mu_i | \cdot |^{(a_i + b_i - 1 - m) / 2}$ is regular algebraic, and so there exists a continuous, semisimple representation $r_\iota(\mu_i | \cdot |^{(a_i + b_i - 1 - m) / 2}) \to \GL_{a_i}(\overline{\bQ}_l)$ satisfying local-global compatibility at each place $w \nmid l$ of $E$ (by e.g.\ \cite{Caraianilnotp}). We can take
	\[ r_\iota(\sigma) = \oplus_{i=1}^k \left( \oplus_{j=1}^{b_i} r_\iota(\mu_i| \cdot |^{(a_i + b_i - 1 - m) / 2}) \otimes \epsilon^{1-j} \right). \]
\end{proof}

\subsection{A comparison of trace formulae}

We continue with the notation and assumptions of the previous section.  We now recall some statements in the theory of the twisted trace formula, following \cite{labesse}. 

Let $f = f^\infty \otimes f_\infty \in C_c^\infty(U_m(\bA_F))$ be a function such that $f_\infty$ is (up to scalar) a pseudocoefficient of a discrete series representation. \cite[Th\'eor\`eme 5.1]{labesse} gives an identity
\begin{equation}\label{eqn_unitary_trace} T^{U_m}_{disc}(f) = \sum_\cE \iota(U_m, \cE) T^{\widetilde{M}^H}_{disc}(\widetilde{f}^H). 
\end{equation}
We explain the notation. The left-hand side $T_{disc}^{U_m}(f)$ equals the trace of $f$ in the space $L^2(U_m(F) \backslash U_m(\bA_F))$; note that $U_m$ is anisotropic, so every automorphic representation is discrete.

The sum on the right-hand side is over isomorphism classes of endoscopic data $\cE$ for $U_m$ (which in our case we can take to consist of a quasi-split reductive group $H$, an element $s \in \widehat{U}_m$, and an $L$-embedding $\eta_H : {}^L H \to {}^L U_m$ such that $\eta_H(\widehat{H}) = \operatorname{Cent}_{\widehat{G}}(s)^\circ$). In the present case the classes of endoscopic data are in bijection with pairs of non-negative integers $(p, q)$ where $p + q = m$ and $p \geq q$; the associated endoscopic group is $H = U^\ast_p \times U^\ast_q$. We specify a representative $\cE_{p, q} = (H, s_H, \eta_H)$ for each isomorphism class. We can take $H = U_p^\ast \times U_q^\ast$. To write down the $L$-embedding $\eta_H$, we first fix a choice of (unitary) character $\mu : \bA_E^\times / E^\times \to \bC^\times$ with the property that $\mu|_{\bA_F^\times / F^\times}$ is the quadratic character associated to the extension $E / F$. If $a \in \bZ$ then we set $\mu_a = 1$ if $a$ is even and $\mu_a = \mu$ if $a$ is odd. We define an $L$-embedding $\eta_H : {}^L H \to {}^L U_m$ by the formulae
\begin{align*} \eta_H : (\GL_p \times \GL_q) \rtimes W_F & \to \GL_m \rtimes W_F  \\
(g_1, g_1) \rtimes 1 & \mapsto  \diag(g_1, g_2) \\
(1_p, 1_q) \rtimes w & \mapsto  \diag(\mu_p(w) 1_p, \mu_q(w) 1_q) \rtimes w \text{ }(w \in W_E) \\
(1_p, 1_q) \rtimes w_c & \mapsto  \diag((-1)^p \Phi_p, (-1)^q \Phi_q) \Phi_m^{-1} \rtimes w_c, \end{align*}
where $w_c \in W_F$ is any fixed lift of $c$. The associated element is $s_H = \eta_H(-1_p, 1_q)$. The constants $\iota(U_m, \cE)$ are given in \cite[Proposition 4.11]{labesse}.

If $H$ is an endoscopic group of $U_m$, then we write $M^H = \Res_{E / F} H_E$, and $\theta_H$ for the automorphism of $M^H$ induced by Galois conjugation. The trace $T^{\widetilde{M}^H}_{disc}(\widetilde{f}^H)$ is the twisted trace of a function $\widetilde{f}^H \in C_c(M^H(\bA_F) \rtimes \theta_H)$ in the twisted-discrete part of $L^2(\mathfrak{A}_{M^H} M^H(F) \backslash M^H(\bA_F))$ (where $\mathfrak{A}_{M^H}$ denotes the connected component of the $\bR$-points of the maximal split subtorus of the centre of $\Res_{F / \bQ} M^H$); see \cite[\S 3.3]{labesse}. 

It remains to define the function $\widetilde{f}^H$. By \cite[Th\'eor\`eme 4.3]{labesse}, the function $f$ admits a transfer $f^H \in C_c^\infty(H(\bA_F))$, which satisfies a certain identity involving the stable orbital integrals of $f^H$. In turn, we may regard $H$ as a principal endoscopic group of the twisted group $\widetilde{M}^H$. By \cite[Lemma 4.1]{labesse} (stable transfer), $f^H$ is associated to a function $\widetilde{f}^H \in C_c^\infty(M^H(\bA_F) \rtimes \theta_H)$. 
This completes our explication of formula (\ref{eqn_unitary_trace}).

Using the expression for $T^{\widetilde{M}^H}_{disc}(\widetilde{f}^H)$ given by \cite[Proposition 3.4]{labesse}, we get a formula
\begin{equation}\label{eqn_twisted_trace} T^{U_m}_{disc}(f) = \sum_
\cE \iota(U_m, \cE) \sum_{L \in \cL^0 / W^{M^H}} \sum_{t \in W^{M^H}(L)_{reg}} \sum_{\widetilde{\pi}^L \in \Pi_{disc}(\widetilde{L}_t)} a^{\widetilde{M}^H}_{disc, \widetilde{L}_t}(\widetilde{\pi}^L) \tr R_Q(\widetilde{\pi}^L)(\widetilde{f}^H). \end{equation}
We do not recall the definition of the terms in detail here, but note that $L$ varies over Levi subgroups of $M^H$ and $\Pi_{disc}(\widetilde{L}_t)$ is the set of automorphic representations of the twisted group $\widetilde{L}_t$ whose restriction to $L$ is irreducible and discrete. 

We now refine the identity (\ref{eqn_twisted_trace}). Fix a finite set $S$ of finite places of $F$ containing all the places above which $\pi_1$ or $\pi_2$ is ramified. We consider test functions of the form $f = f_S \otimes f^{S, \infty} \otimes f_\infty$, where $f_S$ is fixed, $f_\infty$ is a fixed pseudocoefficient of discrete series, and $f^{S, \infty}$ is unramified. As $H$ and $L$ varies, there are only finitely many automorphic representations $\widetilde{\pi}^L$ for which the trace $\tr R_Q(\widetilde{\pi}^L)(\widetilde{f}^H)$ appearing on the right-hand side of (\ref{eqn_twisted_trace}) can be non-zero.

Keeping $f_S$ fixed and using linear independence of twisted characters along with \cite[Proposition 4.9]{labesse}, we can restrict to those summands on each side of the identity  (\ref{eqn_twisted_trace}) which are supported on the twisted character of $\pi^{S, \infty}$. By the results of \cite{MR623137, Moe89}, the only summands which occur on the right-hand side are as follows:
\begin{itemize}
	\item $H = U_m^\ast$, $L = \GL_n \times \GL_n$, $t = \theta_L$, and the restriction of $\widetilde{\pi}^L$ to $L(\bA_F)$ is $\pi_1 \boxtimes \pi_2$ or $\pi_2 \boxtimes \pi_1$.
	\item $H = U_n^\ast \times U_n^\ast$, $L = \GL_n \times \GL_n$, $t = \theta_L$, and the restriction of $\widetilde{\pi}^L$ to $L(\bA_F)$ is $( \pi_1 \boxtimes \pi_2 ) \otimes \mu_n^{-1}$ or $(\pi_2 \boxtimes \pi_1) \otimes \mu_n^{-1}$.
\end{itemize}
We obtain an identity:
\begin{equation}\label{eqn_twisted_trace_refined}
\begin{split} \sum_\sigma m(\sigma) \tr \sigma(f) & = \frac{1}{4} \Big( n(\widetilde{\tau}_0)  \tr (\widetilde{\Pi}_0)(\widetilde{f}^{U^\ast_m}) + n(\widetilde{\tau}_1) \tr (\widetilde{\Pi}_1)(\widetilde{f}^{U^\ast_m})  \\ & +  \tr \widetilde{\tau_0 \otimes \mu_n^{-1}}(\widetilde{f}^{U^\ast_n \times U^\ast_n}) + \tr \widetilde{\tau_1 \otimes \mu_n^{-1}}(\widetilde{f}^{U^\ast_n \times U^\ast_n})\Big). \end{split}
\end{equation}
We explain the terms in (\ref{eqn_twisted_trace_refined}). We set $\tau_0 = \pi_1 \boxtimes \pi_2$ and $\tau_1 = \pi_2 \boxtimes \pi_1$, and $\Pi_0 = \nInd_{\GL_n \times \GL_n}^{\GL_m} \tau_0$ and $\Pi_1 =\nInd_{\GL_n \times \GL_n}^{\GL_m} \tau_1$. Note that $\Pi_0 \cong \Pi_1$. The sum on the left-hand side is over automorphic representations $\sigma$ of $U_m(\A_F)$ such that $\sigma^{S, \infty}$ is related to $\Pi_0^{S, \infty}$ by unramified base change, with $m(\sigma)$ denoting the automorphic multiplicity. On the right-hand side, the twisted trace is taken with respect to the Whittaker normalization of intertwining operators. 

The constants $n(\widetilde{\tau}_i)$ express the difference between Arthur's normalization of the twisted trace on $\widetilde{\Pi}_i$ and the Whittaker normalization (cf. \cite[\S 3.4]{labesse}; it is part of the constant $a^{\widetilde{M}^H}_{disc, \widetilde{L}_t}(\widetilde{\pi}^L)$ of (\ref{eqn_twisted_trace})). In fact, $n(\widetilde{\tau}_0) = n(\widetilde{\tau}_1) = 1$. By symmetry, it is enough to show that $n(\widetilde{\tau}_0) = 1$, and we do this following the argument of \cite[Proposition 4.4.3]{Clo11a}. Let $Q$ denote the standard parabolic subgroup of $\GL_m$ with Levi subgroup $L = \GL_n \times \GL_n$, and let $I_\theta : \Pi_0 \to \Pi_0 \circ \theta_{\GL_m}$ denote Arthur's normalization of the intertwining operator, as described on \cite[p. 437]{labesse}: it is given by a formula
\[ I_\theta = \mathbf{M}_{Q | \theta(Q)} \rho_Q(\theta, 0, \theta). \]
The operator $\mathbf{M}_{Q | \theta(Q)}$ may be expressed, up to analytic continuation, as a product of local unnormalized intertwining operators. Following Shahidi, we can multiply by these local operators by a local factor (defined as in \cite[\S 3.3]{Mok15}) to obtain normalized intertwining operators. By \cite[Proposition 3.5.1]{Mok15}, these are Whittaker normalized (see also \cite[Theorem 2.5.3]{Art13}; note that $\tau_0$ is known to be tempered, by \cite{Clo13, Caraianilnotp}, so this already follows from the results of \cite{Sha90}). The global renormalized  intertwining operator obtained by analytic continuation is
\[ I_\theta' = \epsilon(0, \pi_1 \times \pi_2^\vee) L(\pi_1 \times \pi_2^\vee, 0) L(\pi_1^\vee \times \pi_2, 1) I_\theta = I_\theta, \]
since the Rankin--Selberg $L$-functions appearing here are entire (as $\pi_1$ is not isomorphic to a twist of $\pi_2$ by a power of the norm character, as both are unitary and their sum is regular algebraic) and the functional equation of the Rankin--Selberg $L$-function implies that the global renormalization factor in fact equals 1. This shows that $n(\widetilde{\tau}_0) = 1$. We have obtained an identity
\begin{equation}\label{eqn_simplified_trace_identity}
\begin{split} \sum_\sigma m(\sigma) \tr \sigma(f)  = \frac{1}{4}& \Big( \tr (\widetilde{\Pi}_0)(\widetilde{f}^{U^\ast_m}) + \tr (\widetilde{\Pi}_1)(\widetilde{f}^{U^\ast_m}) \\ & +  \tr \widetilde{\tau_0 \otimes \mu_n^{-1}}(\widetilde{f}^{U^\ast_n \times U^\ast_n}) + \tr \widetilde{\tau_1 \otimes \mu_n^{-1}}(\widetilde{f}^{U^\ast_n \times U^\ast_n})\Big). \end{split}
\end{equation}
To go further we need to introduce rigidifications; more precisely, a normalization of local transfer factors (and a corresponding factorization of e.g.\ $f^{U_n^\ast \times U_n^\ast}$ as a tensor product of associated local terms $f_v^{U_n^\ast \times U_n^\ast}$). Following \cite{Kal14}, we can do this by fixing the following data:
\begin{itemize}
\item A non-trivial additive character $\psi_F : \bA_F / F \to \bC^\times$. (Since $U_m^\ast$ has its standard splitting, it follows that $U_m^\ast(F_v)$ is equipped with a Whittaker datum for each place $v$ of $F$.)
\item A lift of $\xi$ of our fixed inner twist to an extended pure inner twist $(\xi, z)$ in the sense of \cite[\S 0.3]{Kal14}. We suppose $z$ chosen so that its invariants $a_v$, in the sense of \cite[\S 0.3.3]{Kal14}, are given as follows: if $v \in \Sigma$, then $a_v = 2$. If $v | \infty$, then $a_v = n \text{ mod }2$. Otherwise, $a_v = 0$.
\end{itemize}
These choices determine a normalization of local transfer factors satisfying the adelic product formula (see \cite[Proposition 1.1.3]{Kal14} or \cite[Proposition 4.3.2]{Kal19}). This in turn allows us to evaluate the right-hand side of (\ref{eqn_simplified_trace_identity}). More precisely, we first write each twisted trace as a product of local twisted traces. These local twisted traces can be evaluated in terms of stable traces on endoscopic groups of $U_m$, using \cite[Proposition 1.5.2]{Kal14} (a convenient restatement of the results of \cite{Mok15}). These stable traces can then in turned be evaluated in terms of traces on $U_m$ using the endoscopic character identities (thus using again \cite[Proposition 1.5.2]{Kal14} at quasi-split places and \cite[\S 1.6.3]{Kal14} or \cite[Theorem 4.5.1]{Kal19} at places $v \in \Sigma$ or $v | \infty$, respectively). We find that if $\sigma = \otimes'_v \sigma_v$ is the irreducible admissible representation of $U_m(\A_F)$ specified up to isomorphism by the following conditions:
\begin{itemize}
	\item If $v$ is a finite place of $F$ inert in $E$, then $\sigma_v$ is the element of the unramified $L$-packet corresponding to $\pi_v$ such that the character $\langle \sigma_v, \cdot \rangle$ given by \cite[Proposition 1.5.2]{Kal14} is the trivial character;
	\item $\mathrm{BC}(\sigma) = \pi$,
\end{itemize}
then $\sigma$ has multiplicity $m(\sigma) = 1$ as an automorphic representation of $U_m(\A_F)$ provided that the product 
\[ \prod_{v \in \Sigma} \langle \sigma_v, s_{U_n^\ast \times U_n^\ast} \rangle \cdot \prod_{v | \infty} \langle \sigma_v, s_{U_n^\ast \times U_n^\ast} \rangle \]
(these signs now defined by \cite[\S 1.6.3]{Kal14} and \cite[Theorem 4.5.1]{Kal19}) equals 1. We now come to the end of the proof. Indeed, let $F' / F$ be a $\Sigma$-split totally real quadratic extension, let $\psi_{F'} = \psi \circ \tr_{F' / F}$, and let $(\xi', z')$ be the extended pure inner twist $U_{m, F'}^\ast \to U_{m, F'}$ arising by restriction. Then the above analysis goes through as before for the group $U_{m, F'}$ and representation $\pi'$, base change of $\pi$ with respect to $F' / F$. If $\sigma'$ is the irreducible admissible representation of $U_m(\A_{F'})$ defined in the same way relative to $\pi'$ then the multiplicity of $\sigma'$ as an automorphic representation of $U_m(\A_{F'})$ equals 1, provided that the product
\[ \prod_{v \in \Sigma'} \langle \sigma'_v, s_{U_n^\ast \times U_n^\ast} \rangle \cdot \prod_{v | \infty} \langle \sigma'_v, s_{U_n^\ast \times U_n^\ast} \rangle = \left(  \prod_{v \in \Sigma} \langle \sigma_v, s_{U_n^\ast \times U_n^\ast} \rangle \cdot \prod_{v | \infty} \langle \sigma_v, s_{U_n^\ast \times U_n^\ast} \rangle \right)^2 \]
equals 1. Since it is the square of a sign, we're done. 

\section{Congruences between automorphic forms -- unitary group case}\label{sec_congruences_unitary_case}

Let $n \geq 1$ be an integer, and let $m = 2n$. Let $E / F$ be a CM quadratic extension of a totally real number field, and let $D$, $\Sigma$, $B$, $\widetilde{\Sigma}$, $U_m$ etc.\ be as in \S \ref{sec_endoscopic_classification_setup}.

The aim of this section is to prove the following theorem. We fix a prime number $l$, not dividing any element of $\Sigma$, and assume that the $l$-adic places $S_l$ of $F$ split in $E$.
\begin{thm}\label{thm_level_raising_for_unitary_groups}
	Fix an isomorphism $\iota : \overline{\bQ}_l \to \bC$. Let $\sigma$ be an automorphic representation of $U_m(\bA_F)$ with the following properties:
		\begin{enumerate}
			\item $\overline{r}_\iota(\sigma)$ is not isomorphic to a twist of $1 \oplus \epsilon^{-1} \oplus \dots \oplus \epsilon^{1-m}$.
			\item There exists $v_0 \in \Sigma$ and an isomorphism of representations of $\GL_2(D_{\wv_0})$
			\[ \sigma_{v_0} \circ \iota_{\wv_0}^{-1} \cong \nInd_{P_0}^{\GL_2(D_{\wv_0})} \chi_{\wv_0, 1} \circ N \otimes \chi_{\wv_0, 2} \circ N,\]
			 where $\chi_{\wv_0, 1}, \chi_{\wv_0, 2} : E^\times_{\wv_0}\to \bC^\times$ are unramified characters such that 
			 \[ \iota^{-1}(\chi_{\wv_0, 1}(\varpi_{\wv_0}) / \chi_{\wv_0, 2}(\varpi_{\wv_0})) \equiv q_{\wv_0}^{n} \text{ mod }\ffrm_{\overline{\bZ}_l}. \]
		\end{enumerate}
	Then there exists an automorphic representation $\sigma'$ of $U_m(\bA_F)$ with the following properties:
	\begin{enumerate}
		\item $\overline{r}_\iota(\sigma) \cong \overline{r}_\iota(\sigma')$.
		\item There is an isomorphism $\sigma'_{v_0} \circ \iota_{\wv_0}^{-1} \cong \St_{\GL_2(D_{\wv_0})}(\chi_{\wv_0})$, where $\chi_{\wv_0} : E_{\wv_0}^\times \to \bC^\times$ is an unramified character.
		\item $\sigma_\infty \cong \sigma'_\infty$. If $\mathrm{BC}(\sigma)$ is $\iota$-ordinary, then $\mathrm{BC}(\sigma')$ is $\iota$-ordinary.
		\item For each finite place $v$ of $F$ such that $\sigma_v$ is unramified and $v$ is inert in $E$, $\sigma'_v$ is unramified.
	\end{enumerate}
\end{thm}
The proof of Theorem \ref{thm_level_raising_for_unitary_groups} uses algebraic modular forms \cite{MR1729443}. We now define these. Let $W_\infty$ be the irreducible algebraic representation of $(\Res_{F / \bQ} U_{m})_\bC$ such that $\sigma_\infty = W|_{U_m(F \otimes_\bbQ \bR)}$. Then $W_l = \iota^{-1} W_\infty$ is an algebraic representation of $(\Res_{F / \bQ} U_m)_{\overline{\bQ}_l}$, and therefore receives an action of $U_m(F \otimes_\bQ \bQ_l)$.

Let $\cA_\infty$ denote the space of automorphic forms on $U_m(\bA_F)$, and let $\cA_{l, W_l}$ denote the set of  functions $\varphi : U_m(F) \backslash U_m(\bA_F) / U_m(F \otimes_\bQ \bR) \to W_l^\vee$ such that for some open compact subgroup $V \subset U_m(\bA_F^\infty)$, $v_l \varphi(gv) =  \varphi(g)$ for all $g \in U_m(\bA_F)$, $v \in V$. The group $U_m(\bA_F^\infty)$ acts on $\cA_\infty$ by right translation and on $\cA_{l, W_l}$ by the formula $(g \cdot \varphi)(x) = g_l \varphi(x g)$.
\begin{lemma}
	There is a canonical isomorphism of semisimple admissible $\overline{\bQ}_l[U_m(\bA_F^\infty)]$-modules:
	\[ \Hom_{U_m(F \otimes_\bQ \bR)}(W_\infty, \cA_\infty) \otimes_{\bC, \iota^{-1}} \overline{\bQ}_l \cong \cA_{l, W_l}. \]
\end{lemma}
\begin{proof}
	Given $\Phi \in \Hom_{U_m(F \otimes_\bQ \bR)}(W_\infty, \cA_\infty)$, we define $\varphi \in \cA_{l, W_l}$ by 
	\[ \varphi(g)(w) = g_l^{-1} \iota^{-1} g_\infty \Phi(w)(g), \]
	 Conversely, given $\varphi \in \cA_{l, W_l}$, we define $\Phi$ by 
	 \[ \Phi(w)(g) = g_\infty \iota g_l^{-1} \varphi(g)(w). \]
	  It is easy to check that this defines an isomorphism of admissible $\overline{\bQ}_l[U_m(\bA_F^\infty)]$-modules. They are semisimple because $\Hom_{U_m(F \otimes_\bQ \bR)}(W_\infty, \cA_\infty)$ is semisimple (being admissible and unitary).
\end{proof}
Let $K / \bQ_l$ be a finite extension of $\bQ_l$ inside $\overline{\bQ}_l$ containing the image of each embedding $F \to \overline{\bQ}_l$, and let $\cO = \cO_K$, $\lambda = \ffrm_K$, $k = \cO  /\lambda$. Fix for each $v \in S_l$ a factorization $v = \wv \wv^c$ in $E$. Then $W_l$ can be defined over $K$, and we write $M_l \subset W_l$ for a fixed choice of $\cO$-lattice, invariant under the action of the compact subgroup $V_{l, 0} = \prod_{v \in S_l} \iota_\wv^{-1}(\GL_n(\cO_{E_\wv})) \subset U_m(F \otimes_\bQ \bQ_l)$. Let $M_l^\vee = \Hom_\cO(M_l, \cO)$. If $R$ is any $\cO$-algebra, then we let $\cA_{l, M_l, R}$ denote the set of functions 
\[ \varphi : U_m(F) \backslash U_m(\bA_F) / U_m(F \otimes_\bQ \bR) \to M_l^\vee \otimes_\cO R \]
such that for some open compact subgroup $V = V^l V_l \subset U_m(\bA_F^\infty)$ with $V_l \subset V_{l, 0}$, we have $v_l\varphi (gv) =  \varphi(g)$ for all $v \in V$. Then $\cA_{l, M_l, R}$ is an admissible $R[U_m(\bA_F^{\infty, l}) \times V_{l, 0}]$-module. 

Finally, if $V = V^l \times V_l \subset U_m(\bA_F^\infty)$ is any open compact subgroup with $V_l \subset V_{l, 0}$, we define $S(V, R) = \cA_{l, M_l, R}^{V}$. It is a module for the convolution algebra $\cH(U_m(\bA_F^{\infty, l}) \times V_l, V)$ of $V$-biinvariant, compactly supported functions $f : U_m(\bA_F^{\infty, l}) \times V_l \to \bZ$ (defined with respect to the Haar measure giving $V$ measure 1). Henceforth we only consider those choices of $V = V^l V_l$ with $V_l \subset V_{l, 0}$, so we do not make this requirement explicit.
\begin{lemma}\label{lem_sufficiently_small}
\begin{enumerate}
	\item If $R$ is Noetherian then $S(V, R)$ is a finite $R$-module.
	\item Suppose that for each $g \in U_m(\bA_F^\infty)$, the group $g V g^{-1} \cap U_m(F)$ (intersection in $U_m(\bA_F^\infty)$) is trivial. \textup{(}In this case, we say that $V$ is \emph{sufficiently small}.\textup{)} Then there is a canonical isomorphism
	\[ S(V, \cO) \otimes_\cO R \to S(V, R). \]
\end{enumerate}
\end{lemma}
\begin{proof}
	Let $g_1, \dots, g_k$ be a set of representatives for the finite double quotient $U_m(F) \backslash U_m(\bA_F^\infty) / V$. Let $\Gamma_{g_i, V} = g_i^{-1} U_m(F) g_i \cap V$ (intersection in $U_m(\bA_F^\infty))$. Then there is an isomorphism 
	\[  \begin{split} S(V, R) & \cong \oplus_{i=1}^k (M_l^\vee \otimes_\cO R)^{\Gamma_{g_i, V}} \\
	 \varphi & \mapsto (\varphi(g_i))_{i=1}^k. \end{split}  \]
	 If $R$ is Noetherian then the right-hand side is a finitely generated $R$-module. If $V$ is sufficiently small then the right-hand side is a free $R$-module and the canonical map $S(V, \cO) \otimes_\cO R \to S(V, R)$ is an isomorphism.
\end{proof}
We define Hecke algebras using unramified Hecke operators. More precisely, let $V = \prod_v V_v$ be a fixed open compact subgroup of $U_m(\bA_F^\infty)$, and let $T$ be a finite set of places of $F$ satisfying the following conditions:
\begin{itemize}
	\item $T$ contains the infinite places, $S_l$, and $\Sigma$.
	\item For each place $v \not\in T$ of $F$, $V_v$ is a hyperspecial maximal compact subgroup of $U_m(F_v)$.
\end{itemize}
Then we write $\bT^T(V, R)$ for the subalgebra of $\End_R(S(V, R))$ generated by the images of the convolution algebras $\cH(U_m(F_v), V_v)$ for each place $v \not\in T$ of $F$. In particular, if $v = w w^c$ splits in $E$ and $v \not \in T$ then $\bT^T(V, R)$ contains the standard unramified Hecke operators $T_w^1, \dots, T_w^m$, defined as follows: first, we may choose the isomorphism $\iota_w : U_m(F_v) \to \GL_m(E_w)$, defined a priori up to $\GL_m(E_w)$-conjugacy, so that it takes $V_v$ to $\GL_m(\cO_{E_w})$. Then $T_w^i$ is defined to be the Hecke operator which corresponds, under $\iota_w$, to the operator $[\GL_m(\cO_{E_w}) \diag(\varpi_w, \dots, \varpi_w, 1, \dots, 1) \GL_m(\cO_{E_w})]$ with $i$ occurrences of $\varpi_w$ on the diagonal. It is independent of the choice of $\iota_w$. Lemma \ref{lem_sufficiently_small} has the following corollary.
\begin{corollary}
	Let $V = \prod_v V_v$ be a sufficiently small open compact subgroup of $U_m(\bA_F^\infty)$. Then there is a surjective homomorphism $\bT^T(V, \cO) \otimes_\cO k\to \bT^T(V, k)$, which induces a bijection on maximal ideals.
\end{corollary}
\begin{prop}
	For each maximal ideal $\ffrm \subset \bT^T(V, \cO)$, there is an associated Galois representation $\overline{r}_\ffrm : G_E \to \GL_m( \bT^T(V, \cO) / \ffrm)$, uniquely determined up to isomorphism by the following conditions:
	\begin{enumerate}
		\item $\overline{r}_\ffrm$ is semisimple.
		\item For each place $w$ of $E$ not lying above a place of $T$, $\overline{r}_\ffrm|_{G_{E_w}}$ is unramified.
		\item For each place $v = w w^c$ of $F$ split in $E$ such that $v \not\in T$, the characteristic polynomial $\det(X - \overline{r}_\ffrm(\Frob_w))$ equals the image of \[ X^m - T_w^1 X^{m-1} \dots + (-1)^j q_w^{j(j-1)/2} T_w^j + \dots + (-1)^n q_w^{m(m-1)/2} T_w^m. \]
		in $(\T^T(V, \cO)/\ffrm)[X]$.
	\end{enumerate}
\end{prop}
\begin{proof}
This follows from Corollary \ref{cor_existence_of_gal_over_Q_l_for_unitary_group} , \cite[Corollary 3.1.2]{cht}, and the fact that $\bT^T(V, \cO)$ is $\cO$-flat.
\end{proof}
In the statement of the next result, we observe that there is a direct sum decomposition $S(V, k) = \oplus_\ffrm S(V, k)_\ffrm$, the index set being the set of maximal ideals of $\bT^T(V, k)$. We call the $\ffrm$-component of any $f \in S(V, k)$ the image of $f$ under the projection to the factor $S(V, k)_\ffrm$.
\begin{lemma}\label{lem_invariant_implies_non-generic}
	Let $U_m^{der} \subset U_m$ denote the derived subgroup, and let $f \in S(V, k) \subset \cA_{l, M_l, k}$ be invariant under the action of $U_m^{der}(F_{v_0})$, for some $v_0 \in \Sigma$. Then for any maximal ideal $\ffrm \subset \bT^T(V, \cO)$ such that the $\ffrm$-component of $f$ is non-zero, $\overline{r}_\ffrm$ is isomorphic to a twist of $1 \oplus \epsilon^{-1} \oplus \dots \oplus \epsilon^{1-m}$.
\end{lemma}
\begin{proof}
	After modifying $f$, we can assume that $f \in S(V, k)_\ffrm$. After shrinking $V$, we can assume that $V_l$ acts trivially on $M_l \otimes_\cO k$. Then there is an isomorphism $\cA_{l, M_l, k} \cong \cA_{l, \cO, k}^{\dim_k M_l \otimes_\cO k}$, equivariant for the action of the group $U_m(\bA_F^{\infty, l}) \times V_l$. We can therefore assume that $M_l = \cO$ is the trivial representation, in which case we can think of $f$ as a locally constant function $f : U_m(F) \backslash U_m(\bA_F^\infty) / U^{der}_m(F_{v_0}) V^{v_0} \to k$. 
	
	The strong approximation theorem implies that $U_m^{der}(F)$ is dense in $U_m^{der}(\bA_{F}^{\infty, v_0})$. Since the maps $H^1(F, U_m^{der}) \to \prod_{v | \infty} H^1(F_v, U_m^{der})$ and (if $v | \infty$) $H^1(F_v, U_m^{der}) \to H^1(F_v, U_m)$ are both injective (see \cite{Kne69}), consideration of the long exact sequence in Galois cohomology shows that the reduced norm $\det : U_m \to U_1^\ast$  induces a surjection $U_m(F) \to U_1^\ast(F)$. We deduce that $f$ factors through the map
	\[ \det : U_m(F) \backslash U_m(\bA_F^\infty) / V \to U^\ast_1(F) \backslash U^\ast_1(\bA_F^\infty) / \det V. \]
	 Let $G_V = U^\ast_1(F) \backslash U^\ast_1(\bA_F^\infty) / \det V$. Then $G_V$ is a finite abelian group. Let $f' \in \overline{k}[G_V] \cdot f \subset S(V, k)_\ffrm \otimes_k \overline{k}$ be a non-zero vector which spans a simple $\overline{k}[G_V]$-module. Thus there exists a character $\psi : G_V \to \overline{k}^\times$ such that for any $g \in U_m(\bA_F^\infty)$, $g \cdot f' = \psi(\det(g)) f'$. 
	 
	 Let $v = w w^c$ be a place of $F$, split in $E$, such that $v \not\in T$, let $1 \leq i \leq m$, and let $g_{w, i} = \diag(\varpi_w, \dots, \varpi_w, 1, \dots, 1)  \in \GL_m(E_w)$ (where there are $i$ occurrences of $\varpi_w$). Then we have
	 \[ \begin{split} T_w^i f' & = \operatorname{vol}[ \GL_m(\cO_{E_w}) g_{w, i}\GL_m(\cO_{E_w})] \psi(\varpi_w)^i f' \\
& = | 	 \GL_m(\cO_{E_w}) /\big(g_{w, i} \GL_m(\cO_{E_w}) g_{w, i}^{-1} \cap \GL_m(\cO_{E_w})\big)| \psi(\varpi_w)^i f'  \\ 
	 & = q_w^{i(m-i)}\frac{| \GL_m(k(w)) |}{|\GL_i(k(w)) \times \GL_{m-i}(k(w)) |} \psi(\varpi_w)^i f'. \end{split} \]
	 A calculation shows that, if $\alpha_w^i$ denotes the eigenvalue of $T_w^i$ on $f'$, then
	 \[ \prod_{i=1}^m(X - \psi(\varpi_w) q_w^{i-1}) = \sum_{i=0}^m (-1)^i X^{m-i} q^{i(i-1)/2} \alpha_w^i \]
	 (this is essentially the $q$-binomial theorem). It follows from the Chebotarev density theorem that if $\chi : G_E \to \overline{k}^\times$ is the character defined by $\chi \circ \Art_E(x) = \psi(x / x^c)$, then $\overline{r}_\ffrm \cong \chi \otimes( 1 \oplus \epsilon^{-1} \oplus \dots \oplus \epsilon^{1-m})$ (note that the places of $E$ split over $F$ have Dirichlet density 1). This completes the proof.
\end{proof}
If $v \in \Sigma$, then we have fixed an isomorphism $\iota_{\wv} : U_m(F_v) \to \GL_2(D_\wv)$. We set $\mathfrak{K}_v = \iota_\wv^{-1} \GL_2(\cO_{D_\wv})$, and $\mathfrak{I}_v = \iota_\wv^{-1} \mathfrak{I}$, where $\mathfrak{I} \subset \GL_2(\cO_{D_\wv})$ is the standard Iwahori subgroup considered in \S \ref{sec_jacquet_langlands}. We set $\eta_v = \iota_\wv^{-1} \diag(1, \varpi_{D_\wv})$. We write $T_v^1$, $T_v^2 \in \cH(U_m(F_v), \mathfrak{K}_v)$ for the pre-images under $\iota_\wv$ of the operators $T_1, T_2$ which appear in the statement of Lemma \ref{lem_ramified_Hecke_operators}.
\begin{proposition}\label{prop_iharas_lemma}
	Let $v_0 \in \Sigma$, and let $V = \prod_v V_v$ be an open compact subgroup of $U_m(\bA_F^\infty)$ such that $V_{v_0} = \mathfrak{K}_{v_0}$. Let $V' = V^{v_0} \mathfrak{I}_{v_0}$. We define a map
	\[ d : S(V, \cO) \oplus S(V, \cO) \to S(V', \cO) \]
	by the formula $(f, g) \mapsto f + \eta_{v_0} \cdot g$ (action of $\eta_{v_0}$ defined via the inclusion $S(V, \cO) \subset \cA_{l, M_l, \cO}$, cf. Lemma \ref{lem_characterizing_Steinberg}). Then:
	\begin{enumerate}
		\item $d$ is a homomorphism of $\bT^T(V, \cO)$-modules.
		\item Suppose that $V$ is sufficiently small, and that $\ffrm \subset \bT^T(V, \cO)$ is a maximal ideal such that $\overline{r}_\ffrm$ is not isomorphic to a twist of $1 \oplus \epsilon^{-1} \oplus \dots \oplus \epsilon^{1-m}$. Then the induced homomorphism 
		\[ d_\ffrm : S(V, \cO)_\ffrm \oplus S(V, \cO)_\ffrm \to S(V', \cO)_\ffrm \]
		is injective and has saturated image (i.e. its cokernel is torsion-free).
		\item Suppose that $V$ is sufficiently small, and that $\ffrm \subset \bT^T(V, \cO)$ is a maximal ideal such that $\overline{r}_\ffrm$ is not isomorphic to a twist of $1 \oplus \epsilon^{-1} \oplus \dots \oplus \epsilon^{1-m}$. Suppose there exists $f \in S(V, k)_\ffrm - \{  0 \}$ be such that 
		\[ [ (T^1_{v_0})^2 - T^2_{v_0} (q_{v_0}^n + 1)^2 ] f = 0. \] Then $d_\ffrm$ is not an isomorphism.
	\end{enumerate}
\end{proposition}
\begin{proof}
	The first point is immediate, because $v_0 \in T$. For the second it is enough (using Lemma \ref{lem_sufficiently_small}) to show that the map
	\[ d_{\ffrm, k} : S(V, k)_\ffrm \oplus S(V, k)_\ffrm \to S(V', k)_\ffrm \]
	given by the same formula $d_{\ffrm, k}(f, g) = f + \eta_{v_0} \cdot g$ is injective. By Lemma \ref{lem_kernel_of_d}, the kernel of $d_{\ffrm, k}$ is contained in the subspace of $\cA_{l, M_l, k}$ where $U_m^{der}(F_{v_0})$ acts trivially. By Lemma \ref{lem_invariant_implies_non-generic} and our hypothesis on $\ffrm$, the intersection of this subspace with $S(V, k)_\ffrm$ is zero, and $d_{\ffrm, k}$ is indeed injective.
	
	We now prove the third point. It is enough to show that $d_{\ffrm, k}$ is not an isomorphism. If $R$ is an $\cO$-algebra, let $S^\vee(V, R)$ denote the space of modular forms defined in the same way as $S(V, R)$, but with $M_l$ replaced by $M_l^\vee$. (Thus the elements of $S^\vee(V, R)$ are functions $\varphi : U_m(F) \backslash U_m(\bA_F) / U_m(F \otimes_\bQ \bR) \to M_l \otimes_\cO R$.) We can define an $R$-linear pairing 
	\[ \begin{split} \langle \cdot, \cdot \rangle_V& :  S(V, R) \times S^\vee(V, R) \to R \\
	\langle \varphi, \varphi' \rangle_V & = \sum_{g \in U_m(F) \backslash U_m(\bA_F^\infty) / V} \langle \varphi(g), \varphi'(g) \rangle. \end{split} \]
	The adjoint of a Hecke operator $[V g V]$ acting on $S(V, R)$ with respect to this pairing is $[V g^{-1} V]$. Let $\bT^{T, \vee}(V, \cO)$ denote the $\cO$-subalgebra of $\End_\cO(S^\vee(U, \cO))$ generated by the Hecke operators away from $T$. It follows that there is a maximal ideal $\ffrm^\vee \subset \bT^{T, \vee}(V, \cO)$ such that $\langle \cdot, \cdot \rangle_V$ restricts to a perfect $\cO$-linear pairing
	\[ \langle \cdot, \cdot \rangle_{V, \ffrm} : S(V, \cO)_\ffrm \times S^\vee(V, \cO)_{\ffrm^\vee} \to \cO, \]
	and moreover $\overline{r}_{\ffrm^\vee} \cong \overline{r}_\ffrm^c$.
	Applying Lemma \ref{lem_sufficiently_small}, this pairing in turn gives a perfect $k$-linear pairing
	\[  \langle \cdot, \cdot \rangle_{V, \ffrm, k} : S(V, k)_\ffrm \times S^\vee(V, k)_{\ffrm^\vee} \to k. \]
	We can consider the map
	\[ d_{\ffrm^\vee, k} : S^\vee(V, k)_{\ffrm^\vee} \oplus S^\vee(V, k)_{\ffrm^\vee} \to S^\vee(V', k)_{\ffrm^\vee}. \]
	Its adjoint, computed with respect to the pairings $\langle \cdot, \cdot \rangle_{V, \ffrm, k}$ and $\langle \cdot, \cdot \rangle_{V', \ffrm, k}$, is a map
	\[ d_{\ffrm^\vee, k}^\vee : S(V', k)_\ffrm \to S(V, k)_\ffrm \oplus S(V, k)_\ffrm. \]
	The second part of the proposition applies equally well to $d_{\ffrm^\vee, k}$, showing that $d^\vee_{\ffrm^\vee, k}$ is surjective. We see that $d_{\ffrm, k}$ is an isomorphism exactly when its source and target have the same dimension, which happens only if $d_{\ffrm^\vee, k}^\vee \circ d_{\ffrm, k}$ is an isomorphism. However, a computation (essentially the same one appearing in \cite[Lemma 2]{Tay89}) shows that $d_{\ffrm^\vee, k}^\vee \circ d_{\ffrm, k}$ equals the matrix of Hecke operators
	\[ d_{\ffrm^\vee, k}^\vee \circ d_{\ffrm, k} = \left( \begin{array}{cc} q_{v_0}^n + 1 & T^1_{v_0} \\ T^1_{v_0} (T^2_{v_0})^{-1} & q_{v_0}^n + 1 \end{array}\right).  \]
	 It follows that the determinant of $d_{\ffrm^\vee, k}^\vee \circ d_{\ffrm, k}$ as a $k$-linear endomorphism of $S(V, k)_\ffrm^2$ equals the determinant of $(q^n_{v_0} + 1)^2 - (T_{v_0}^1)^2 (T_{v_0}^2)^{-1}$ as a $k$-linear endomorphism of $S(V, k)_\ffrm$. This completes the proof.
\end{proof}
We can now give the proof of Theorem \ref{thm_level_raising_for_unitary_groups}. Let $\sigma$ be an automorphic representation of $U_m(\bA_F)$, and fix $v_0 \in \Sigma$ such that $\sigma_{v_0} \circ \iota^{-1}_{\wv_0} \cong \nInd_{P_0}^{\GL_2(D_{\wv_0})} \chi_1 \otimes \chi_2$, where $\chi_1, \chi_2 : E_{\wv_0}^\times \to \bC^\times$ are unramified characters such that $\iota^{-1}(\chi_1(\varpi_{\wv_0}) / \chi_2(\varpi_{\wv_0})) \equiv q_{\wv_0}^{n} \text{ mod }\ffrm_{\overline{\bbZ}_l}$. Then $\sigma_{v_0}^{\mathfrak{K}_{v_0}} \neq 0$. Let $V = \prod_v V_v \subset U_m(\bA_F^\infty)$ be an open compact subgroup satisfying the following conditions:
\begin{itemize}
	\item $(\sigma^\infty)^V \neq 0$.
	\item $V_l \subset V_{l, 0}$.
	\item $V_{v_0} = \mathfrak{K}_{v_0}$.
	\item If $v \not\in \Sigma \cup S_l$ is a finite inert place of $F$ such that $\sigma_v$ is unramified, then $V_v$ is a hyperspecial maximal compact subgroup of $U_m(F_v)$ such that $\sigma_v^{V_v} \neq 0$.
	\item $V$ is sufficiently small. 
\end{itemize}
	Then there is a natural inclusion $\iota^{-1}(\sigma^\infty)^V \subset \cA_{l, W_l}^V$. Let $T$ be a finite set of places such that the Hecke algebra $\bT^T(V, \cO)$ is defined, let $\mathfrak{p} \subset \bT^T(V, \cO)$ be the kernel of the natural homomorphism $\bT^T(V, \cO) \to \End_{\overline{\bQ}_l}(\iota^{-1} (\sigma^\infty)^V)$, and let $\ffrm \subset \bT^T(V, \cO)$ be the unique maximal ideal containing $\mathfrak{p}$. Let $V' \subset U_m(\bA_F^\infty)$ be the group associated to $V$ as in the statement of Proposition \ref{prop_iharas_lemma}. Proposition \ref{prop_iharas_lemma} implies that the map 
	\[ d_\ffrm : S(V, \cO)_\ffrm \oplus S(V, \cO)_\ffrm \to S(V', \cO)_\ffrm \]
	 is not an isomorphism. Indeed, the eigenvalue of $(T^1_{v_0})^2 - T^2_{v_0} (q^n_{v_0} + 1)^2$ on $\sigma^{V_{v_0}}_{v_0}$ is, by Lemma \ref{lem_ramified_Hecke_operators}, an element of $\ffrm_{\overline{\bZ}_l}$, yet this eigenvalue is also a root of the characteristic polynomial of this Hecke operator acting on $S(V, \cO)_\ffrm$, so $(T^1_{v_0})^2 - T^2_{v_0} (q^n_{v_0} + 1)^2$ must have a non-trivial kernel in $S(V, k)_\ffrm$. The second part of the proposition shows that $d_\ffrm$ is injective with saturated image, so it follows that the induced map
	\[ (  S(V, \cO)_\ffrm \oplus S(V, \cO)_\ffrm ) \otimes_\cO \overline{\bQ}_l \to S(V', \cO)_\ffrm \otimes_\cO \overline{\bQ}_l \]
	is not surjective. Since $\cA_{l, W_l}$ is a semisimple $\overline{\bQ}_l[U_m(\bA_F^\infty)]$-module, this implies that there is an automorphic representation $\sigma'$ of $U_m(\bA_F^\infty)$ with the following properties:
	\begin{itemize}
		\item $(\sigma^{\prime, \infty})^{V'} \neq 0$.
		\item $\overline{r}_\iota(\sigma') \cong \overline{r}_\ffrm \cong \overline{r}_\iota(\sigma)$.
		\item $(\sigma'_{v_0})^{\mathfrak{I}_{v_0}} \neq 0$, yet the map $d_{\sigma_{v_0}} : (\sigma'_{v_0})^{\mathfrak{K}_{v_0}} \oplus (\sigma'_{v_0})^{\mathfrak{K}_{v_0}} \to (\sigma'_{v_0})^{\mathfrak{I}_{v_0}}$ is not surjective. 
	\end{itemize}
	Lemma \ref{lem_characterizing_Steinberg} implies that $\sigma'_{v_0}$ is a twist of $\St_{\GL_2(D_{\wv_0})}$ by an unramified character. This completes the proof, except we still need to explain why $\mathrm{BC}(\sigma')$ can be chosen to be $\iota$-ordinary if $\mathrm{BC}(\sigma)$ is. This can be achieved by enlarging the Hecke algebra $\bT^T(V, \cO)$ to include the Hecke operators $U_{\boldsymbol{\lambda}, v}^j$ for $v \in S_l$ (cf. \cite[\S 2.4]{Clo14} or \cite[\S 2.4]{ger}); indeed, the ordinary subspace of $S(V, \cO)$ can be defined as the maximal direct summand $\cO$-module where each of these operators acts invertibly. We omit the routine modifications required to the proof.
\section{Congruences between automorphic forms -- general linear group case}\label{sec_congruences_general_linear_case}

We can now prove the main theorem of this paper. 
\begin{theorem}\label{thm_level_raising_for_general_linear_groups_two_factor_case}
	Let $n \geq 1$ be an integer and let $E$ be a CM number field. Let $F$ be the maximal totally real subfield of $E$, and assume that $E / F$ is everywhere unramified. Let $l$ be a prime, and fix an isomorphism $\iota : \overline{\bQ}_l \to \bC$. Let $w_0$ be a prime-to-$l$ place of $E$ which splits over $F$. Let $\pi_1, \pi_2$ be cuspidal, conjugate self-dual automorphic representations of $\GL_n(\bA_{E})$ satisfying the following conditions:
	\begin{enumerate}
		\item $\pi = \pi_1 \boxplus \pi_2$ is regular algebraic. 
		\item For any place $w$ of $E$, if $\pi_w$ is ramified then $w$ is split over $F$. The $l$-adic places of $F$ all split in $E$.
		\item There are isomorphisms $\pi_{i, w_0} \cong \St_n(\xi_{i})$ for some unramified characters $\xi_{i} : E_{w_0}^\times \to \bC^\times$ \textup{(}$i = 1, 2$\textup{)}, and 
		\[ \iota^{-1} \xi_{1}(\varpi_{w_0}) / \xi_{2}(\varpi_{w_0}) \equiv q_{w_0}^{n} \text{ mod }\ffrm_{\overline{\bZ}_l}. \]
		\item $\overline{r}_\iota(\pi)$ is not isomorphic to a twist of $1 \oplus \epsilon^{-1} \oplus \dots \oplus \epsilon^{1-2n}$.
	\end{enumerate}
Let $F' / F$ be any totally real, quadratic $w_0$-split extension, and let $E' = E F'$. Suppose further that $\overline{r}_\iota(\pi)|_{G_{E'}}$ is not isomorphic to a twist of $1 \oplus \epsilon^{-1} \oplus \dots \oplus \epsilon^{1-2n}$. Then there exists a RACSDC automorphic representation $\Pi$ of $\GL_{2 n}(\bA_{E'})$ and a place $w_0' | w_0$ of $E'$ satisfying the following conditions:
	\begin{enumerate}
		\item There is an isomorphism $\overline{r}_\iota(\Pi) \cong \overline{r}_\iota(\pi)|_{G_{E'}}$.
		\item There is an isomorphism $\Pi_{w'_0} \cong \St_{2n}(\xi')$, where $\xi' : E_{w'_0}^\times \to \bC^\times$ is an unramified character. 
		\item For each archimedean place $v'$ of $E$ lying above a place $v$ of $E$, $\Pi_{v'}$ and $\pi_{v}$ have the same infinitesimal character. If $\pi$ is $\iota$-ordinary, then $\Pi$ is $\iota$-ordinary. 
		\item For any place $w$ of $E'$, if $\Pi_w$ is ramified then $w$ is split over $F'$.
	\end{enumerate}
\end{theorem}
\begin{proof}
	Let $v_0$ denote the place of $F$ lying below $w_0$ and let $\Sigma = \{ v_0\}$. Let $m = 2n$ and choose a unitary group $U_{m}$ as in \S \ref{sec_endoscopic_classification_setup}. Theorem \ref{thm_unitary_group_descent} implies the existence of an automorphic representation $\sigma$ of $U_m(\bA_{F'})$ such that for each inert place $v$ of $F'$, $\sigma_v$ is unramified, and $\operatorname{BC}(\sigma) = \pi_{E'}$. This automorphic representation satisfies the hypotheses of Theorem \ref{thm_level_raising_for_unitary_groups}, with $\Sigma'$ equal to the set of places of $F'$ lying above $\Sigma$. Using this theorem we can find an automorphic representation $\sigma'$ of $U_m(\bA_{F'})$ with the following properties:
	\begin{itemize}
		\item $\overline{r}_\iota(\sigma') \cong \overline{r}_\iota(\pi)|_{G_{E'}}$.
		\item There exists $v_0' \in \Sigma'$ and an isomorphism $\sigma'_{v'_0}\circ \iota_{w'_0}^{-1} \cong \St_{\GL_2(D_{w'_0})}(\xi')$, where $\xi' : E'_{w'_0} \to \bC^\times$ is an unramified character and $w_0'$ is the unique place of $E'$ lying above both $v'_0$ and $w_0$.
		\item $\sigma_\infty \cong \sigma'_\infty$. $\mathrm{BC}(\sigma')$ is $\iota$-ordinary if $\mathrm{BC}(\sigma)$ is. 
		\item For each place finite place $v$ of $F'$ which is inert in $E'$, $\sigma'_v$ is unramified.
	\end{itemize}
	Theorem \ref{thm_unitary_group_base_change} implies the existence of a regular algebraic automorphic representation $\Pi$ of $\GL_m(\bA_{E'})$ such that $\operatorname{BC}(\sigma') = \Pi$. To complete the proof, it remains only to justify why $\Pi_{w_0'}$ is an unramified twist of the Steinberg representation (then $\Pi$ is necessarily cuspidal, and satisfies all of the other requirements). We have $| LJ_{\GL_2(D_{w'_0})}|(\Pi_{w'_0}) = \sigma'_{v'_0} \circ \iota_{w_0'}$, an unramified twist of the Steinberg representation of $\GL_2(D_{w'_0})$. 
	
	By the third part of Lemma \ref{lem_computation_of_LJ}, $\Pi_{w'_0}$ is (up to unramified twist) either the Steinberg representation or the trivial representation. If $\Pi_{w'_0}$ is the trivial representation then the classification of the discrete spectrum of $\GL_m(\bA_{E'})$ (\cite{Moe89}) implies that $\Pi$ is itself 1-dimensional, contradicting our assumption on $\overline{r}_\iota(\Pi)$. This completes the proof.  
\end{proof}

\bibliographystyle{amsalpha}
\bibliography{CMpatching}

\renewcommand{\MR}[1]{}
\providecommand{\bysame}{\leavevmode\hbox to3em{\hrulefill}\thinspace}
\providecommand{\MR}{\relax\ifhmode\unskip\space\fi MR }
\providecommand{\MRhref}[2]{%
  \href{http://www.ams.org/mathscinet-getitem?mr=#1}{#2}
}
\providecommand{\href}[2]{#2}
\begin{thebibliography}{MgW89}

\bibitem[AC89]{MR1007299}
James Arthur and Laurent Clozel, \emph{Simple algebras, base change, and the
  advanced theory of the trace formula}, Annals of Mathematics Studies, vol.
  120, Princeton University Press, Princeton, NJ, 1989. \MR{MR1007299
  (90m:22041)}

\bibitem[Art13]{Art13}
James Arthur, \emph{The endoscopic classification of representations}, American
  Mathematical Society Colloquium Publications, vol.~61, American Mathematical
  Society, Providence, RI, 2013, Orthogonal and symplectic groups. \MR{3135650}

\bibitem[Bad07]{Bad07}
Alexandru~Ioan Badulescu, \emph{Jacquet-{L}anglands et unitarisabilit\'{e}}, J.
  Inst. Math. Jussieu \textbf{6} (2007), no.~3, 349--379. \MR{2329758}

\bibitem[Bad08]{Bad08}
\bysame, \emph{Global {J}acquet-{L}anglands correspondence, multiplicity one
  and classification of automorphic representations}, Invent. Math.
  \textbf{172} (2008), no.~2, 383--438, With an appendix by Neven Grbac.
  \MR{2390289}

\bibitem[BG06]{Bel06}
Jo\"{e}l Bella\"{\i}che and Phillipe Graftieaux, \emph{Augmentation du niveau
  pour {${\rm U}(3)$}}, Amer. J. Math. \textbf{128} (2006), no.~2, 271--309.
  \MR{2214894}

\bibitem[Car12]{Caraianilnotp}
Ana Caraiani, \emph{Local-global compatibility and the action of monodromy on
  nearby cycles}, Duke Math. J. \textbf{161} (2012), no.~12, 2311--2413.
  \MR{2972460}

\bibitem[Cas80]{Cas80}
W.~Casselman, \emph{The unramified principal series of {${\mathfrak{p}}$}-adic
  groups. {I}. {T}he spherical function}, Compositio Math. \textbf{40} (1980),
  no.~3, 387--406. \MR{571057}

\bibitem[Cas81]{Cas81}
\bysame, \emph{A new nonunitarity argument for {$p$}-adic representations}, J.
  Fac. Sci. Univ. Tokyo Sect. IA Math. \textbf{28} (1981), no.~3, 907--928
  (1982). \MR{656064}

\bibitem[CHL11]{Clo11a}
Laurent Clozel, Michael Harris, and Jean-Pierre Labesse, \emph{Endoscopic
  transfer}, On the stabilization of the trace formula, Stab. Trace Formula
  Shimura Var. Arith. Appl., vol.~1, Int. Press, Somerville, MA, 2011,
  pp.~475--496. \MR{2856382}

\bibitem[CHT08]{cht}
Laurent Clozel, Michael Harris, and Richard Taylor, \emph{Automorphy for some
  {$l$}-adic lifts of automorphic mod {$l$} {G}alois representations}, Publ.
  Math. Inst. Hautes \'{E}tudes Sci. (2008), no.~108, 1--181, With Appendix A,
  summarizing unpublished work of Russ Mann, and Appendix B by Marie-France
  Vign\'{e}ras. \MR{2470687}

\bibitem[Clo00]{Clo00}
L.~Clozel, \emph{On {R}ibet's level-raising theorem for {$\rm U(3)$}}, Amer. J.
  Math. \textbf{122} (2000), no.~6, 1265--1287. \MR{1797662}

\bibitem[Clo13]{Clo13}
Laurent Clozel, \emph{Purity reigns supreme}, Int. Math. Res. Not. IMRN (2013),
  no.~2, 328--346. \MR{3010691}

\bibitem[CT14]{Clo14}
Laurent Clozel and Jack~A. Thorne, \emph{Level-raising and symmetric power
  functoriality, {I}}, Compos. Math. \textbf{150} (2014), no.~5, 729--748.
  \MR{3209793}

\bibitem[DKV84]{Del84}
P.~Deligne, D.~Kazhdan, and M.-F. Vign\'{e}ras, \emph{Repr\'{e}sentations des
  alg\`ebres centrales simples {$p$}-adiques}, Representations of reductive
  groups over a local field, Travaux en Cours, Hermann, Paris, 1984,
  pp.~33--117. \MR{771672}

\bibitem[DT94]{Dia94}
Fred Diamond and Richard Taylor, \emph{Nonoptimal levels of mod {$l$} modular
  representations}, Invent. Math. \textbf{115} (1994), no.~3, 435--462.
  \MR{1262939}

\bibitem[Gee11]{Gee11}
Toby Gee, \emph{Automorphic lifts of prescribed types}, Math. Ann. \textbf{350}
  (2011), no.~1, 107--144. \MR{2785764}

\bibitem[Ger19]{ger}
David Geraghty, \emph{Modularity lifting theorems for ordinary {G}alois
  representations}, Math. Ann. \textbf{373} (2019), no.~3-4, 1341--1427.
  \MR{3953131}

\bibitem[Gro99]{MR1729443}
Benedict~H. Gross, \emph{Algebraic modular forms}, Israel J. Math. \textbf{113}
  (1999), 61--93. \MR{MR1729443 (2001b:11037)}

\bibitem[HKV]{Her19}
Florian Herzig, Karol Koziol, and Marie-France {Vign\'eras}, \emph{On the
  existence of admissible supersingular representations of $p$-adic reductive
  groups}, arXiv preprint. With an appendix by Sug Woo Shin.

\bibitem[JS81]{MR623137}
H.~Jacquet and J.~A. Shalika, \emph{On {E}uler products and the classification
  of automorphic forms. {II}}, Amer. J. Math. \textbf{103} (1981), no.~4,
  777--815. \MR{MR623137 (82m:10050b)}

\bibitem[Kal16]{Kal16c}
Tasho Kaletha, \emph{The local {L}anglands conjectures for non-quasi-split
  groups}, Families of automorphic forms and the trace formula, Simons Symp.,
  Springer, [Cham], 2016, pp.~217--257. \MR{3675168}

\bibitem[Kar]{Kar19}
Aditya Karnataki, \emph{Level-raising for automorphic forms on {$GL(n)$}},
  Preprint.

\bibitem[KMSW]{Kal14}
Tasho Kaletha, Alberto Minguez, Sug~Woo Shin, and Paul-James White,
  \emph{Endoscopic classification of representations: Inner forms of unitary
  groups}, arXiv preprint.

\bibitem[Kne69]{Kne69}
M.~Kneser, \emph{Lectures on {G}alois cohomology of classical groups}, Tata
  Institute of Fundamental Research, Bombay, 1969, With an appendix by T. A.
  Springer, Notes by P. Jothilingam, Tata Institute of Fundamental Research
  Lectures on Mathematics, No. 47. \MR{0340440}

\bibitem[Kot88]{Kot88}
Robert~E. Kottwitz, \emph{Tamagawa numbers}, Ann. of Math. (2) \textbf{127}
  (1988), no.~3, 629--646. \MR{942522}

\bibitem[KT]{Kal19}
Tasho Kaletha and Olivier Ta{\"i}bi, \emph{Global rigid forms vs isocrystals},
  Preprint.

\bibitem[Lab11]{labesse}
J.-P. Labesse, \emph{Changement de base {CM} et s\'{e}ries discr\`etes}, On the
  stabilization of the trace formula, Stab. Trace Formula Shimura Var. Arith.
  Appl., vol.~1, Int. Press, Somerville, MA, 2011, pp.~429--470. \MR{2856380}

\bibitem[M\'11]{Min11}
Alberto M\'{i}nguez, \emph{Unramified representations of unitary groups}, On
  the stabilization of the trace formula, Stab. Trace Formula Shimura Var.
  Arith. Appl., vol.~1, Int. Press, Somerville, MA, 2011, pp.~389--410.
  \MR{2856377}

\bibitem[MgW89]{Moe89}
C.~M\oe~glin and J.-L. Waldspurger, \emph{Le spectre r\'{e}siduel de {${\rm
  GL}(n)$}}, Ann. Sci. \'{E}cole Norm. Sup. (4) \textbf{22} (1989), no.~4,
  605--674. \MR{1026752}

\bibitem[Mok15]{Mok15}
Chung~Pang Mok, \emph{Endoscopic classification of representations of
  quasi-split unitary groups}, Mem. Amer. Math. Soc. \textbf{235} (2015),
  no.~1108, vi+248. \MR{3338302}

\bibitem[MS06]{Sor06}
Claus Mazanti~Sorensen, \emph{A generalization of level-raising congruences for
  algebraic modular forms}, Ann. Inst. Fourier (Grenoble) \textbf{56} (2006),
  no.~6, 1735--1766. \MR{2282674}

\bibitem[NM43]{Nak43}
Tadasi Nakayama and Yoz\^{o} Matsushima, \emph{\"{U}ber die multiplikative
  {G}ruppe einer {$p$}-adischen {D}ivisionsalgebra}, Proc. Imp. Acad. Tokyo
  \textbf{19} (1943), 622--628. \MR{0014081}

\bibitem[NT]{New19b}
James Newton and Jack~A. Thorne, \emph{Symmetric power functoriality for
  holomorphic modular forms}, Preprint.

\bibitem[Rib84]{Rib84}
Kenneth~A. Ribet, \emph{Congruence relations between modular forms},
  Proceedings of the {I}nternational {C}ongress of {M}athematicians, {V}ol. 1,
  2 ({W}arsaw, 1983), PWN, Warsaw, 1984, pp.~503--514. \MR{804706}

\bibitem[Sha90]{Sha90}
Freydoon Shahidi, \emph{A proof of {L}anglands' conjecture on {P}lancherel
  measures; complementary series for {$p$}-adic groups}, Ann. of Math. (2)
  \textbf{132} (1990), no.~2, 273--330. \MR{1070599}

\bibitem[Tad90]{Tad90}
Marko Tadi\'{c}, \emph{Induced representations of {${\rm GL}(n,A)$} for
  {$p$}-adic division algebras {$A$}}, J. Reine Angew. Math. \textbf{405}
  (1990), 48--77. \MR{1040995}

\bibitem[Tay89]{Tay89}
Richard Taylor, \emph{On {G}alois representations associated to {H}ilbert
  modular forms}, Invent. Math. \textbf{98} (1989), no.~2, 265--280.
  \MR{1016264}

\bibitem[Tay08]{tay}
\bysame, \emph{Automorphy for some $l$-adic lifts of automorphic mod $l$
  {G}alois representations. {II}}, Pub. Math. IHES \textbf{108} (2008),
  183--239.

\bibitem[Tho14]{Tho14}
Jack~A. Thorne, \emph{Raising the level for {${\rm GL}_{n}$}}, Forum Math.
  Sigma \textbf{2} (2014), e16, 35. \MR{3264255}

\end{thebibliography}

\end{document}